\documentclass[10pt]{amsart}
\usepackage{amsthm, amsfonts, amssymb, amsmath, oldgerm}

\newtheorem{theorem}{Theorem}[section]
\newtheorem{proposition}[theorem]{Proposition}

\newtheorem{lemma}[theorem]{Lemma}

\newtheorem{remark}[theorem]{Remark}

\newtheorem*{remark*}{Remark}
\numberwithin{equation}{section}

\renewcommand{\geq}{\geqslant}
\renewcommand{\leq}{\leqslant}

\title[Stability of periodic waves]
{Stability of periodic traveling waves \\ for nonlinear dispersive equations}

\author[Hur]{Vera~Mikyoung~Hur}
\address{Department of Mathematics, University of Illinois at Urbana-Champaign, Urbana, IL 61801 USA}
\email{verahur@math.uiuc.edu}

\author[Johnson]{Mathew~A.~Johnson}
\address{Department of Mathematics, University of Kansas, Lawrence, KS 66045 USA} 
\email{matjohn@ku.edu}  

\date{\today}

\keywords{stability; periodic traveling waves; nonlinear dispersive; nonlocal}
\subjclass[2010]{35B35, 35Q53, 35B10}

\begin{document}

\begin{abstract}
We study the stability and instability of periodic traveling waves for Korteweg-de Vries type equations with fractional dispersion and related, nonlinear dispersive equations. We show that a local constrained minimizer for a suitable variational problem is nonlinearly stable to period preserving perturbations, provided that the associated linearized operator enjoys a Jordan block structure. We then discuss when the linearized equation admits solutions exponentially growing in time. 
\end{abstract}

\maketitle

\section{Introduction}\label{S:intro}
We study the stability and instability of periodic traveling waves 
for a class of nonlinear dispersive equations, 
in particular, equations of Korteweg-de Vries (KdV) type
\begin{equation}\label{E:KdV1} 
u_t-\mathcal{M}u_x+f(u)_x=0.
\end{equation}
Here $t \in \mathbb{R}$ denotes the temporal variable and 
$x \in \mathbb{R}$ is the spatial variable in the predominant direction of wave propagation;
$u=u(x,t)$ is real valued, representing the wave profile or a velocity. 
Throughout we express partial differentiation either by a subscript or using the symbol $\partial$. 
Moreover $\mathcal{M}$ is a Fourier multiplier, 
defined as $\widehat{\mathcal{M}u}(\xi)=m(\xi)\widehat{u}(\xi)$
and characterizing dispersion in the linear limit, while $f$ is the nonlinearity. 
In many examples of interest, $f$ obeys a power law. 

Perhaps the best known among equations of the form \eqref{E:KdV1} is the KdV equation 
\[
u_t+u_{xxx}+(u^2)_x=0
\] 
itself, which was put forward in \cite{Bsnesq} and \cite{KdV} 
to model the unidirectional propagation of surface water waves 
with small amplitudes and long wavelengths in a channel; 
it has since found relevances in other situations such as Fermi-Pasta-Ulam lattices. 
Observe, however, that \eqref{E:KdV1} is {\em nonlocal} 
unless the dispersion symbol $m$ is a polynomial of $i\xi$;
examples include the Benjamin-Ono equation (see \cite{Benjamin, Ono}, for instance) 
and the intermediate long wave equation (see \cite{Joseph}, for instance), 
for which $m(\xi)=|\xi|$ and $\xi\coth\xi-1$, respectively, while $f(u)=u^2$. 
Another example, proposed by Whitham \cite{Whitham} to argue for breaking of water waves, 
corresponds to $m(\xi)=\sqrt{(\tanh\xi)/\xi}$ and $f(u)=u^2$. 
Incidentally the quadratic nonlinearity is characteristic of many wave phenomena. 

\

A traveling wave solution of \eqref{E:KdV1} takes the form $u(x,t)=u(x-ct)$, 
where $c\in\mathbb{R}$ and $u$ satisfies by quadrature that 
\[
\mathcal{M}u-f(u)+cu+a=0
\]
for some $a\in \mathbb{R}$. 
In other words, it steadily propagates at a constant speed without changing the configuration. 
Periodic traveling waves of the KdV equation are known in closed form, 
namely cnoidal waves; see \cite{KdV}, for instance. 
Moreover Benjamin \cite{Benjamin} calculated periodic traveling waves of the Benjamin-Ono equation. 
For a broad range of dispersion symbols and nonlinearities,
a plethora of periodic traveling waves of \eqref{E:KdV1} may be attained from variational arguments. 
To illustrate, we shall discuss in Section~\ref{S:existence} 
a minimization problem for a family of KdV equations with fractional dispersion.

Benjamin in his seminal work \cite{Ben-KdV} (see also \cite{Bona}) explained that 
solitary waves of the KdV equation are nonlinearly stable. 
By a solitary wave, incidentally, we mean a traveling wave solution which vanishes asymptotically. 
Benjamin's proof hinges upon that 
the KdV ``soliton" arises as a constrained minimizer for a suitable variational problem 
and spectral information of the associated linearized operator. 
Later it developed into a powerful stability theory in \cite{GSS}, for instance, 
for a general class of Hamiltonian systems and led to numerous applications. 
In the case of $m(\xi)=|\xi|^\alpha$, $\alpha\geq 1$, and $f(u)=u^{p+1}$, $p\geq 1$, in \eqref{E:KdV1}, 
in particular, solitary waves were shown in \cite{BSS} (see also \cite{SS, Ws}) 
to arise as energy minimizers subject to the conservation of the momentum 
and to be nonlinearly stable if $p<2\alpha$
whereas they are constrained energy saddles and nonlinearly unstable if $p>2\alpha$.

We shall take matters further in Section~\ref{S:stability} and establish that 
a periodic traveling wave of a KdV equation with fractional dispersion 
is nonlinearly stable with respect to period preserving perturbations, 
provided that it locally minimizes the energy subject to conservations of the momentum and the mass
and that the associated linearized operator enjoys a Jordan block structure.
Moreover we relate the latter condition with the momentum and the mass 
as functions of Lagrange multipliers arising in the traveling wave equation,
generalizing that in \cite{BSS}, for instance, in the solitary wave setting.
In the case of generalized KdV equations, i.e., $m(\xi)=\xi^2$ in \eqref{E:KdV1}, 
the nonlinear stability of a periodic traveling wave to same period perturbations 
was determined in \cite{J1}, for instance, through spectral conditions,
which were expressed in terms of eigenvalues of the associated monodromy map 
(or the periodic Evans function); see also \cite{P,PBSM, BJK, DK, DN}. 
Confronted with nonlocal operators, however, 
spectral problems may be out of reach by Evans function techniques.  
Instead we make an effort to replace ODE based arguments by functional analytic ones.
The program was recently set out in \cite{BHV}.

As a key intermediate step we shall demonstrate in Section \ref{S:nondegeneracy} that 
the linearized operator associated with the traveling wave equation is {\em nondegenerate} 
at a periodic, local constrained minimizer for a KdV equation with fractional dispersion. 
That is to say, its kernel is spanned merely by spatial translations. 
The nondegeneracy of the linearization proves a spectral condition, 
which plays a central role in the stability of traveling waves (see \cite{Ws, Lin1} among others) 
and the blowup (see \cite{KMR}, for instance) for the related, time evolution equation, 
and therefore it is of independent interest. In the case of generalized KdV equations, 
the nondegeneracy at a periodic traveling wave was identified in \cite{J1}, for instance, 
with that the wave amplitude not be a critical point of the period. 
Furthermore it was verified in \cite{Kwong}, among others, at solitary waves. 
These proofs utilize shooting arguments and the Sturm-Liouville theory for ODEs, 
which may not be applicable to nonlocal operatorors. 
Nevertheless, Frank and Lenzmann \cite{FL} obtained the property at solitary waves 
for a family of nonlinear nonlocal equations, which we follow. 
The idea lies in to find a suitable substitute for the Sturm-Liouville theory 
to count the number of sign changes in eigenfunctions 
for a linear operator with a fractional Laplacian. 

\

The present development may readily be adapted to other, nonlinear dispersive equations. 
We shall illustrate this in Section \ref{S:BBM} by discussing equations of regularized long wave type. 
We shall remark in Section \ref{S:instability} about Lin's approach \cite{Lin1} to linear instability. 

\section{Existence of local constrained minimizers}\label{S:existence}
We shall address the stability and instability mainly for the KdV equation 
with fractional dispersion 
\begin{equation}\label{E:KdV}
u_t-\Lambda^{\alpha} u_x+(u^2)_x=0,
\end{equation}
where $0<\alpha\leq 2$ and $\Lambda=\sqrt{-\partial_x^2}$ is defined via the Fourier transform 
as $\widehat{\Lambda u}(\xi)=|\xi|\hat{u}(\xi)$. 

In the case of $\alpha=2$, notably, \eqref{E:KdV} recovers the KdV equation,
and in the case of $\alpha=1$ it corresponds to the Benjamin-Ono equation. In the case\footnote{
Note that $\Lambda^\alpha\partial_x$ is non singular for $\alpha\geq -1$.} 
of $\alpha=-1/2$, furthermore, \eqref{E:KdV} was argued in \cite{Hur-breaking} 
to have relevances to surface water waves in two dimensions in the infinite depths.
Observe that \eqref{E:KdV} is nonlocal for $0<\alpha<2$. 
Incidentally fractional powers of the Laplacian occur in numerous applications, 
such as dislocation dynamics in crystals (see \cite{CFM}, for instance) 
and financial mathematics (see \cite{CT}, for instance). 

The present treatment extends mutatis mutandis to general power-law nonlinearities;
see Remark \ref{R:gKdV}. We focus on the quadratic nonlinearity, however, to simplify the exposition. 

\

Throughout we'll work in the $L^2$-based Sobolev spaces over the periodic interval $[0,T]$, 
where $T>0$ is fixed although at times it is treated as a free parameter. For $0<\alpha<2$ let
\[ 
\|u\|_{H^{\alpha/2}_{per}([0,T])}^2=\int^T_0(u^2+|\Lambda^{\alpha/2}u|^2)~dx.
\]
We employ the standard notation $\langle\cdot\,,\cdot\rangle$ for 
the $L^2_{per}([0,T])$-inner product. 

\

Notice that \eqref{E:KdV} may be written in the Hamiltonian form
\begin{equation}\label{E:H}
u_t=J\delta H(u),
\end{equation}
where $J=\partial_x$ is the symplectic form,
\begin{equation}\label{E:HKdV}
H(u)=\int_0^T\Big(\frac{1}{2}|\Lambda^{\alpha/2}u|^2-\frac{1}{3}u^{3}\Big)~dx=:K(u)+U(u)
\end{equation}
is the Hamiltonian and $\delta$ denotes variational differentiation;
$K$ and $U$ correspond to the kinetic and potential energies, respectively. 
Notice that \eqref{E:KdV} possesses, in addition to $H$, two conserved quantities 
\begin{align}
P(u)=&\int_0^T\frac{1}{2}u^2~dx \label{E:PKdV}
\intertext{and}
M(u)=&\int_0^Tu~dx, \label{E:MKdV}
\end{align}
which correspond to the momentum and the mass, respectively. 
Conservation of $P$ implies that \eqref{E:KdV} is invariant under spatial translations 
thanks to Noether's theorem while $M$ is a Casimir invariant of the flow induced by \eqref{E:KdV} 
and is associated with that the kernel of the symplectic form is spanned by a constant. 
Notice that  
\begin{equation}\label{E:PM}
\delta P(u)=u \quad\text{and}\quad \delta M(u)=1.
\end{equation}
Moreover \eqref{E:KdV} remains invariant under 
\begin{equation}\label{E:scaling}
u(x,t)\mapsto \lambda^\alpha u(\lambda (x-x_0), \lambda^{\alpha+1}t)
\end{equation}
for any $\lambda>0$ for any $x_0\in\mathbb{R}$.

\begin{remark}[Well-posedness]\label{R:LWP}\rm
In the range $\alpha\geq-1$, one may work out the local in time well-posedness 
for \eqref{E:KdV} in $H^{3/2+}_{per}([0,T])$, combining an a priori bound and a compactness argument. 
Without recourse to dispersive effects, 
the proof is identical to that for the inviscid Burgers equation, i.e., $\alpha=0$. We omit the detail. 

With the help of techniques in nonlinear dispersive equations and specific properties of the equation, 
the global in time well-posedness for \eqref{E:KdV} may be established in $H^{-1/2+}_{per}([0,T])$
in the case of $\alpha=2$, namely the KdV equation (see \cite{CKSTT}, for instance), 
and in $H^{0+}_{per}([0,T])$ in the case of $\alpha=1$, 
the Benjamin-Ono equation (see \cite{Molinet}, for instance). 
For non-integer values of $\alpha$, however, 
the existence matter for \eqref{E:KdV} seems not adequately understood in spaces of low regularities. 
The global well-posedness in $H^{\alpha/2}(\mathbb{R})$ was recently settled in \cite{KMR} 
for \eqref{E:KdV}, in the case of $1<\alpha< 2$ and $u^{p+1}$ in place of $u^2$, $p\leq 2\alpha$,
but the proof seems to break down in the periodic functions setting.  
\end{remark}

In what follows we shall work in a suitable subspace, say, $X$ of $H^{\alpha/2}_{per}([0,T])$, 
where the initial value problem associated with \eqref{E:KdV} is well-posed for some interval of time 
and $H, P, M: X \to \mathbb{R}$ are smooth.

\

A periodic traveling wave of \eqref{E:KdV} takes the form $u(x,t)=u(x-ct-x_0)$, 
where $c\in\mathbb{R}$ represents the wave speed, $x_0\in\mathbb{R}$ is the spatial translate 
and $u$ is $T$-periodic, satisfying by quadrature that
\begin{equation}\label{E:pKdV}
\Lambda^\alpha u-u^{2}+cu+a=0
\end{equation}
for some $a\in\mathbb{R}$ (in the sense of distributions). Equivalently, it arises as a critical point of 
\begin{equation}\label{D:E}
E(u;c,a)=H(u)+cP(u)+aM(u).
\end{equation}
Indeed 
\begin{equation}\label{E:pKdV'}
\delta E(u;c,a)=0 
\end{equation}
agrees with \eqref{E:pKdV}. 

Henceforth we shall write a periodic traveling wave of \eqref{E:KdV} as $u=u(\cdot\,;c,a)$. 
In a more comprehensive description, it is specified by four parameters $c$, $a$ and $T$, $x_0$. 
Note, however, that $T>0$ is arbitrary and fixed. 
Corresponding to translational invariance (see \eqref{E:scaling}), moreover, 
$x_0$ is inconsequential in the present development. Hence we may mod it out. 

In the present notation, a solitary wave whose profile vanishes asymptotically corresponds,
formally, to $a=0$ and $T=+\infty$.  

\

In the case of $\alpha=2$, periodic traveling waves of \eqref{E:KdV}, namely the KdV equation, 
are well known in closed form, involving Jacobi elliptic functions; see \cite{KdV}, for instance. 
In the case of $\alpha=1$, moreover, 
Benjamin \cite{Benjamin} exploited the Poisson summation formula 
and explicitly calculated periodic traveling waves of \eqref{E:KdV}. 
In general, the existence of periodic traveling waves of \eqref{E:KdV} 
follows from variational arguments, although one may lose an explicit form of the solution. 
In the energy subcritical case, in particular, 
a family of periodic traveling waves of \eqref{E:KdV} locally minimizes the Hamiltonian 
subject to conservations of the momentum and the mass, 
generalizing ``ground states" in the solitary wave setting. 

\begin{proposition}[Existence, symmetry and regularity]\label{P:existence}
Let $1/3<\alpha\leq 2$.
A local minimizer $u$ for $H$ subject to that $P$ and $M$ are conserved 
exists in $H^{\alpha/2}_{per}([0,T])$ for each $0<T<\infty$ 
and it satisfies \eqref{E:pKdV} for some $c\neq 0$ and $a \in \mathbb{R}$. 
It depends upon $c$ and $a$ in the $C^1$ manner.

Moreover $u=u(\cdot\,; c,a)$ may be chosen to be even and strictly decreasing over the 
interval $[0,T/2]$, and $u \in H^\infty_{per}([0,T])$. 
\end{proposition}

Below we develop integral identities which a periodic solution of \eqref{E:pKdV},
or equivalently \eqref{E:pKdV'}, a priori satisfies and which will be useful in various proofs. 

\begin{lemma}[Integral identities]
If $u \in H^{\alpha/2}_{per}([0,T]) \cap L^3_{per}([0,T])$ satisfies \eqref{E:pKdV}, 
or equivalently \eqref{E:pKdV'}, then 
\begin{gather}
2P-cM-aT=0, \label{E:I1-KdV} \\ 
2K+3U+2cP+aM=0. \label{E:I2-KdV}
\end{gather}
\end{lemma} 
 
\begin{proof}
Integrating \eqref{E:pKdV}, or equivalently \eqref{E:pKdV'}, over the periodic interval $[0,T]$ 
leads to \eqref{E:I1-KdV}. Multiplying it by $u$ and integrating over $[0,T]$ lead to \eqref{E:I2-KdV}.
\end{proof}
 
\begin{proof}[Proof of Proposition~\ref{P:existence}]
We claim that it suffices to take $a=0$ and $c=1$. Suppose on the contrary that $a\neq 0$.
We then assume without loss of generality that $c$ and $M$ are of opposite sign and $a> 0$. 
For, in case $c$ and $M$ are of the same sign, since \eqref{E:KdV} is time reversible, 
we make the change of variables $t \mapsto -t$ in \eqref{E:KdV} 
to reverse the sign of $c$ in \eqref{E:pKdV} while leaving other components of the equation invariant. 
Once we accomplish that $c$ and $M$ are of opposite sign, $a\geq 0$ must follow
since $P\geq 0$ and $T>0$ by definition. 
We shall then devise the change of variables $u \mapsto u+\frac12 (\sqrt{c^2+4a}-c)$ 
and rewrite \eqref{E:pKdV} as
\begin{equation}\label{E:a=0} 
\Lambda^\alpha u-u^2+\gamma u=0,\qquad \text{where}\quad \gamma =\sqrt{c^2+4a}>0.
\end{equation}
Therefore it suffices to take $a=0$ in \eqref{E:pKdV}. 
This is reminiscent of that \eqref{E:KdV} enjoys Galilean invariance 
under $u(x,t)\mapsto u(x, t)+u_0$ for any $u_0 \in \mathbb{R}$. 
By virtue of scaling invariance (see \eqref{E:scaling}), 
we shall further devise the change of variables $u(x) \mapsto 1/\gamma u(x/\gamma^\alpha)$ 
and rewrite \eqref{E:a=0} as
\begin{equation}\label{E:a=0,c=1}
\Lambda^\alpha u-u^2+u=0.
\end{equation}
To recapitulate, it suffices to take $a=0$ and $c=1$ in \eqref{E:pKdV} 
and seek a local minimizer for $H+P$. 
(But we shall not  a priori assume that $a=0$ or $c=1$ in the stability proof in Section \ref{S:stability}.) 

\

Since $H^{\alpha/2}_{per}([0,T])$ in the range $\alpha>1/3$ is compactly embedded 
in $L^3_{per}([0,T])$ by a Sobolev inequality, it follows from calculus of variations that 
for each parameter (abusing notation) $U<0$\footnote{
Note from \eqref{E:I2-KdV} that if $u\in H^{\alpha/2}_{per}([0,T])$, $\alpha>1/3$, 
satisfies \eqref{E:a=0,c=1} then $K(u)+P(u)>0$ and $U(u)<0$ unless $u\equiv 0$. }
there exists $u\in H^{\alpha/2}_{per}([0,T])$ such that 
\begin{equation}\label{E:K+P}
K(u)+P(u)=\inf \big\{K(\phi)+P(\phi): \phi \in H^{\alpha/2}_{per}([0,T]),\,U(\phi)=U \big\}.
\end{equation}
The proof is rudimentary. We merely pause to remark that 
$K(\phi)+P(\phi)$ amounts to $\|\phi\|_{H^{\alpha/2}_{per}([0,T])}^2$
and the constraint is compact in $H^{\alpha/2}_{per}([0,T])$. Moreover, $u$ satisfies
\begin{equation*}\label{E:dist}
\Lambda^\alpha u+ u=\theta u^2
\end{equation*}
for some $\theta\neq 0$ in the sense of distributions. 
By a scaling argument, we may choose $U$ to ensure that $\theta=1$.
Consequently (abusing notation) $u\in H^{\alpha/2}_{per}([0,T])$ attains 
the constrained minimization problem \eqref{E:K+P} and satisfies \eqref{E:a=0,c=1}.
Note from \eqref{E:I2-KdV} that $2K(u)+3U(u)+2P(u)=0$.

Furthermore we claim that 
\begin{equation}\label{E:H+P}
E(u)=\inf\{ E(\phi): \phi\in H^{\alpha/2}_{per}([0,T]),\,
\phi\not\equiv 0,\, 2K(\phi)+3U(\phi)+2P(\phi)=0\}.
\end{equation}
Since 
\begin{equation}\label{E:ident}
E(\phi)=H(\phi)+P(\phi)=K(\phi)+U(\phi)+P(\phi)=\frac13(K(\phi)+P(\phi))=-\frac12U(\phi)
\end{equation}
and $2K(\phi)+2P(\phi)=-3U(\phi)>0$ whenever $2K(\phi)+3U(\phi)+2P(\phi)=0$, $\phi\not\equiv0$, 
it suffices to show that 
\begin{equation}\label{E:U}
U(u)=\sup\{U(\phi): \phi \in H^{\alpha/2}_{per}([0,T]),\, 
\phi\not\equiv 0, \, 2K(\phi)+3U(\phi)+2P(\phi)=0\}.
\end{equation}
Suppose that $\phi \in H^{\alpha/2}_{per}([0,T])$, $\phi\not\equiv 0$ and
$2K(\phi)+3U(\phi)+2P(\phi)=0$. We define 
\begin{equation*}\label{E:b}
b=\left(\frac{U(u)}{U(\phi)}\right)^{1/3},
\end{equation*}
and observe that \eqref{E:U} follows if $b\leq 1$ so that $0\geq U(u)> U(\phi)$.
Indeed we infer from \eqref{E:ident} that
\begin{align*}
2K(b\phi)+3U(b\phi)+2P(b\phi)=&2b^2K(\phi)+3b^3U(\phi)+2b^2P(\phi) \\
=&2b^2(1-b)(K(\phi)+P(\phi)). 
\end{align*}
Moreover, since $U(b\phi)=b^3U(\phi)=U(u)$ and 
since $u$ attains the constrained minimization problem \eqref{E:K+P}, it follows that
\[
K(u)+P(u)\leq K(b\phi)+P(b\phi).
\] 
Consequently
\begin{align*}
0=2K(u)+3U(u)+2P(u)\leq &2K(b\phi)+3U(b\phi)+2P(b\phi) \\
=&2b^2(1-b)(K(\phi)+P(\phi)),
\end{align*}
whence $b\leq 1$. This proves the claim. Since 
\[
\left\langle \delta H(\phi)+\delta P(\phi), \phi\right\rangle=2K(\phi)+3U(\phi)+2P(\phi)
\]
for all $\phi\in H^{\alpha/2}_{per}([0,T])$, furthermore, 
$u$ solves the constrained minimization problem \eqref{E:H+P} if and only if
$u$ minimizes $H+P$ among its critical points. The existence assertion therefore follows. 
Clearly, $u$ depends upon $c$ and $a$ in the  $C^1$ manner.

\

To proceed, since the symmetric decreasing rearrangement of $u$ 
does not increase $\int^T_0 |\Lambda^{\alpha/2}u|^2~dx$ for $0<\alpha<2$ 
(see \cite{YP}, for instance, for a proof in the solitary wave setting) 
while leaving $\int_0^Tu^3~dx$ invariant,
it follows from the rearrangement argument that 
a local minimizer for $H$ subject to conservations of $P$ and $M$ 
must symmetrically decrease away from a point of principal elevation. 
The symmetry and monotonicity assertion then follows 
from translational invariance in \eqref{E:scaling}. 
(Note that unlike in the solitary waves setting, for which $a=0$ and $T=+\infty$, 
a periodic, local constrained minimizer needs not be positive everywhere.)

\

It remains to address the smoothness of a periodic solution of \eqref{E:pKdV}, or equivalently, 
\begin{equation}\label{E:integralE} 
u=(\Lambda^\alpha+1)^{-1}u^2
\end{equation}
after reduction to $a=0$, $c=1$ and after inversion.
The validity of \eqref{E:integralE} is to be specified in the course the proof. 
We claim that  if $u \in H^{\alpha/2}_{per}([0,T])$ satisfies \eqref{E:integralE} 
then $u \in L^\infty_{per}([0,T])$. 
In the case of $\alpha>1$ this follows immediately from a Sobolev inequality, 
whereas in the case of $1/3<\alpha\leq 1$ a proof based upon resolvent bounds for
$(\Lambda^\alpha+1)^{-1}$ is found in \cite[Lemma A.3]{FL}, for instance, 
albeit in the solitary wave setting.  Indeed, the Fourier series $\widehat{\frac{1}{|n|^\alpha+1}}$ 
lies in $\ell^r(\mathbb{Z})$ for $0< \alpha<1$ for $r>\frac{1}{1-\alpha}$ by the Hausdorff-Young inequality, 
whence $u \in L^\infty_{per}([0,T])$ after iterating \eqref{E:integralE} sufficiently many times. 

We then promote $u \in H^{\alpha/2}_{per}([0,T])\cap L^\infty_{per}([0,T])$ to $H^\alpha_{per}([0,T])$ 
since the Plancherel theorem leads to that
\[
\|\Lambda^\alpha u\|_{L^2}=\Big\|\frac{\Lambda^\alpha}{\Lambda^\alpha+1}u^2\Big\|_{L^2}
=\Big\|\frac{|\xi|^\alpha}{|\xi|^\alpha+1}\widehat{u^2}\Big\|_{L^2}\leq \|\widehat{u^2}\|_{L^2}
=\|u^2\|_{L^2} \leq \|u\|_{L^\infty}\|u\|_{L^2}<\infty.
\]
Furthermore the fractional product rule (see \cite{CW}, for instance) leads to that 
\[ 
\|\Lambda^{2\alpha} u\|_{L^2}=\Big\|\frac{\Lambda^{2\alpha}}{\Lambda^\alpha+1}u^2\Big\|_{L^2}
\leq \|\Lambda^\alpha u^2\|_{L^2} \leq C\|u\|_{L^\infty}\|\Lambda^\alpha u\|_{L^2}<\infty
\]
for $C>0$ a constant independent of $u$. After iterations, therefore, $u\in H^\infty_{per}([0,T])$ follows.
\end{proof}

\begin{remark}[Power-law nonlinearities]\label{R:gKdV}\rm
One may rerun the arguments in the proof of Proposition \ref{P:existence} 
in the case of the general power-law nonlinearity
\begin{equation}\label{E:KdV'}
u_t -\Lambda^\alpha u_x+(u^{p+1})_x=0
\end{equation}
and obtain a periodic traveling wave, 
where $0<\alpha\leq 2$ and $0<p<p_{max}$ is an integer such that 
\begin{equation}\label{D:pmax}
p_{max}:=\begin{cases} 
\frac{2\alpha}{1-\alpha} &\text{for $\alpha<1$,} \\ 
+\infty &\text{for $\alpha\geq 1$.}
\end{cases}\end{equation}
It locally minimizes in $H^{\alpha/2}_{per}([0,T])$ the Hamiltonian
\[ 
\int^T_0 \Big(\frac12|\Lambda^{\alpha/2} u|^2-\frac{1}{p+2}u^{p+2}\Big)~dx
\] 
subject to conservations of $P$ and $M$, defined in \eqref{E:PKdV} and \eqref{E:MKdV}, 
respectively. Note that $0<p<p_{max}$, which is vacuous if $\alpha\geq 1$, ensures that 
\eqref{E:KdV'} is $H^{\alpha/2}$-subcritical and
$H^{\alpha/2}_{per}([0,T]) \subset L^{p+2}_{per}([0,T])$ compactly. 
In the case of $p=1$, it is equivalent to that $\alpha>1/3$. 
\end{remark}

\begin{remark}[Periodic vs. solitary waves]\label{R:solitary}\rm
In the non-periodic functions setting, Weinstein \cite{Ws} (see also \cite{FL}) proved that 
\eqref{E:pKdV} in the range $\alpha>1/3$ admits a solitary wave, for which $a=0$ and $T=+\infty$. 
In the case of $\alpha>1/2$ so that \eqref{E:pKdV} is $L^2$-subcritical,
the solitary wave further arises as an energy minimizer subject to the conservation of the momentum. 
Periodic, local constrained minimizers for \eqref{E:pKdV}, constructed in Proposition \ref{P:existence}, 
are then expected to tend to the solitary wave as their period increases to infinity. 
This in some sense generalizes the homoclinic limit in the case of $\alpha=2$.  

In the case of $1/3<\alpha<1/2$, on the other hand, local constrained minimizers for \eqref{E:pKdV} exist 
in the periodic wave setting, but they are unlikely to achieve a limiting state
with bounded energy (the $H^{\alpha/2}$-norm) as the period increases to infinity. 

In the $L^2$-critical case, i.e. $\alpha=1/2$, periodic traveling waves with small energy
tend to the solitary wave as their period increases to infinity. 
Their stability is, however, delicate and outside the scope of the present development.
We refer the reader to \cite{KMR}, for instance.
\end{remark}

For a broad range of dispersion operators and nonlinearities, 
including $\alpha\geq -1$ in \eqref{E:KdV},
one is able to construct periodic traveling waves of \eqref{E:KdV1} at least with small amplitudes 
from perturbation arguments such as the Lyapunov-Schmidt reduction; see \cite{HJ2}, for instance.
In the solitary wave setting, in stark contrast, Pohozaev identities techniques dictate that 
\eqref{E:pKdV} ($a=0$) in the range $\alpha\leq1/3$ does not admit 
any nontrivial solutions in $H^{\alpha/2}(\mathbb{R})\cap L^3(\mathbb{R})$. 


\section{Nondegeneracy of the linearization}\label{S:nondegeneracy}
Throughout the section, let $u(\cdot\,; c,a)$ be a periodic traveling wave of \eqref{E:KdV}, 
whose existence follows from Proposition~\ref{P:existence}. 
We shall examine the $L^2_{per}([0,T])$-null spaces 
of the linearizations associated with \eqref{E:pKdV} and \eqref{E:KdV}.

\begin{proposition}[Nondegeneracy]\label{P:nondegeneracy}
Let $1/3<\alpha\leq 2$. If $u(\cdot\,;c,a) \in H^{\alpha/2}_{per}([0,T])$ 
for some $c\neq 0$, $a\in \mathbb{R}$ and for some $T>0$ 
locally minimizes $H$ subject to that $P$ and $M$ are conserved 
then the associated linearized operator 
\begin{equation}\label{E:L+KdV}
\delta^2E(u; c,a)=\Lambda^\alpha-2u+c
\end{equation}
acting on $L^2_{per}([0,T])$ is non\-de\-gen\-er\-ate. That is to say, 
\[
\ker(\delta^2E(u; c,a))={\rm span}\{u_x\}.
\]
\end{proposition}

The nondegeneracy of the linearization associated with the traveling wave equation is 
of paramount importance in the stability of traveling waves and the blowup 
for the related, time evolution equation; see \cite{Ws, Lin1, KMR}, among others. 
To prove the property is far from being trivial, however. 
Actually, one may cook up a polynomial nonlinearity, say, $f$, for which 
the kernel of $-\partial_x^2-f'(u)$ at a periodic traveling wave $u$ 
is two dimensional at isolated points. 

In the case of generalized KdV equations, 
for which $\alpha=2$ in \eqref{E:KdV} but the nonlinearity is arbitrary, 
the nondegeneracy of the linearization at a periodic traveling wave was shown in \cite{J1}, 
for instance, to be equivalent to that the wave amplitude not be a critical point of the period;
the proof uses the Sturm-Liouville theory for ODEs.
Furthermore it was verified in \cite{Kwong}, among others, at solitary waves (in all dimensions). 
Amick and Toland \cite{AT} demonstrated the property in the case of $\alpha=1$ in \eqref{E:KdV}, 
namely the Benjamin-Ono equation, 
in the periodic and solitary wave settings, by relating via complex analysis techniques 
the nonlocal, traveling wave equation to a fully nonlinear ODE; 
unfortunately, the arguments are specific to the equation. 
Angulo Pava and Natali \cite{A-P} made an alternative proof 
based upon the theory of totally positive operators, but it necessitates an explicit form of the solution. 
A satisfactory understanding of the nondegeneracy of the linearization
thus seems largely missing for nonlocal equations. 
The main obstruction is that shooting arguments and other ODE methods, 
which seem crucial in the arguments for local equations, may not be applicable.

Nevertheless, Frank and Lenzmann \cite{FL} recently obtained the nondegeneracy of the linearization
at solitary waves for a family of nonlinear nonlocal equations with fractional derivatives. 
Their idea is to find a suitable substitute for the Sturm-Liouville theory 
to estimate the number of sign changes in eigenfunctions for a fractional Laplacian with potential. 
Our proof of Proposition~\ref{P:nondegeneracy} follows along the same line 
as the arguments in \cite[Section 3]{FL}, but with appropriate modifications 
to accommodate the periodic nature of the problem. 

\begin{lemma}[Oscillation of eigenfunctions]\label{L:nodal}
Under the hypothesis of Proposition~\ref{P:nondegeneracy}, 
an eigenfunction in $H^{\alpha/2}_{per}([0,T]) \cap C^0_{per}([0,T])$ 
corresponding to the $j$-th eigenvalue of $\delta^2E(u)$, $j=1,2,3$,  
changes its sign at most $2(j-1)$ times over the periodic interval $[0,T]$.
\end{lemma}

We shall present the proof in Appendix~\ref{S:appendix}.

\begin{remark}[Oscillation of higher eigenfunctions]\label{R:oscillation}\rm
Lemma~\ref{L:nodal} holds for all $j=1,2,3,\dots$.
See \cite{HJM}, where the proof and applications are studied.  
\end{remark}


Below we gather some facts about $\delta^2E(u)$.

\begin{lemma}[Properties of $\delta^2E(u)$]\label{L:L+}
Under the hypothesis of Proposition~\ref{P:nondegeneracy}, the followings hold:
\begin{itemize}
\item[(L1)] $u_x\in\ker(\delta^2E(u))$ and it corresponds to the lowest eigenvalue of $\delta^2E(u)$
restricted to the sector of odd functions in $L^2_{per}([0,T])$;
\item[(L2)] $1\leq n_-(\delta^2E(u))\leq 2$, where $n_-(\delta^2E(u))$ means 
the number of negative eigenvalues of $\delta^2E(u)$ acting on $L^2_{per}([0,T])$;
\item[(L3)] $1,u, u^2\in {\rm range}(\delta^2E(u))$.
\end{itemize} 
\end{lemma}

\begin{proof} 
Differentiating \eqref{E:pKdV} with respect to $x$ implies that $\delta^2E(u)u_x=0$. 
Moreover Proposition~\ref{P:existence} implies that 
$u$ may be chosen to satisfy $u_x(x)<0$ for $0<x<T/2$. 
The lowest eigenvalue of $\delta^2E(u)$ 
acting on the sector of odd functions in $L^2_{per}([0,T])$, denoted $L^2_{per, odd}([0,T])$, 
on the other hand, must be simple and
the corresponding eigenfunction is strictly positive (or negative) over the half interval $[0,T/2]$; 
a proof based upon the Perron-Frobenious argument is rudimentary and hence we omit the detail. 
Therefore zero is the lowest eigenvalue of $\delta^2E(u)$ restricted to $L^2_{per, odd}([0,T])$ 
and $u_x$ is a corresponding eigenfunction. 

\

To proceed, recall that $u_x$ belongs to the kernel of $\delta^2E(u)$ 
and attains zero twice over the periodic interval $[0,T]$. 
Since an eigenfunction associated with the lowest eigenvalue of $\delta^2E(u)$ 
is strictly positive (or negative), $\delta^2E(u)$ acting on $L^2_{\rm per}([0,T])$ 
must have at least one negative eigenvalue.

Moreover, since $u$ locally minimizes $H$, and hence $E$, 
subject to conservations of $P$ and $M$, necessarily, 
\begin{equation}\label{E:2negative} 
\delta^2E(u)|_{\{\delta P(u), \delta M(u)\}^\perp} \geq 0.
\end{equation}
This implies by Courant's mini-max principle that 
$\delta^2E(u)$ has at most two negative eigenvalues, asserting (L2). 

\

Lastly, differentiating \eqref{E:pKdV'} with respect to $c$ and $a$, respectively, 
we use \eqref{E:PM} to obtain that
\begin{equation}\label{E:Lu}
\delta^2E(u)u_c=-\delta P(u)=-u\quad \text{and}\quad \delta^2E(u)u_a=-\delta M(u)=-1.
\end{equation}
Therefore $1,u\in \text{range}(\delta^2E(u))$. Incidentally
\begin{align}\label{E:McPa}
M_c(u(\cdot;c,a))&=\langle\delta M(u),u_c\rangle=\langle-\delta^2E(u)u_a,u_c\rangle\\
&=\langle u_a,-\delta^2E(u)u_c\rangle=\langle u_a,\delta P(u)\rangle=P_a(u(\cdot;c,a)). \notag
\end{align}
Since 
\[
\delta^2E(u)u=\Lambda^\alpha u-2u^2+cu=-u^2-a
\]
by \eqref{E:pKdV}, moreover, $u^2\in\text{range}(\delta^2E(u))$. \end{proof}

\begin{proof}[Proof of Proposition \ref{P:nondegeneracy}]
Consider the orthogonal decomposition 
\[
L^2_{per}([0,T])=L^2_{per, odd}([0,T])\oplus L^2_{per, even}([0,T]).
\] 
Since $u$ may be chosen to be even by Proposition~\ref{P:existence}, it follows that 
$L^2_{per, odd}([0,T])$ and $L^2_{per, even}([0,T])$ are invariant subspaces of $\delta^2E(u)$. 
Since (L1) of Lemma~\ref{L:L+} implies that 
\[
\ker(\delta^2E(u)|_{L^2_{per, odd}([0,T])})=\text{span}\{u_x\},
\]
moreover, it remains to show that $\ker(\delta^2E(u)|_{L^2_{per, even}([0,T])})=\{0\}$. 

Suppose on the contrary that 
there were a nontrivial function $\phi \in L^2_{per, even}([0,T])$ such that $\delta^2E(u)\phi=0$. 
Since $\delta^2E(u)$ has at most two negative eigenvalues by (L2) of Lemma \ref{L:L+}, 
it follows from Lemma \ref{L:nodal} that 
$\phi$ changes its sign at most twice over the half interval $[0,T/2]$. 
Consequently, unless $\phi$ is positive (or negative) throughout the periodic interval $[0,T]$, 
either there exists $T_1\in(0,T/2)$ such that $\phi$ is positive (or negative) for $0<|x|<T_1$ 
and negative (or positive, respectively) for $x \in (-T/2,T_1)\cup(T_1,T/2)$, 
or there exist $T_1<T_2$ in $[0, T/2)$ 
such that $\phi$ is positive for $|x|<T_1$ and $T_2<|x|<T/2$ 
(with the understanding that the first interval is empty in case $T_1=0$) 
and $\phi$ is negative for $x\in (-T_2,-T_1)\cup(T_1,T_2)$. 

Since $\phi$ lies in the kernel of $\delta^2E(u)$, 
on the other hand, it must be orthogonal to $\text{range}(\delta^2E(u))$ 
and, in turn, to the subspace $\text{span}\{1,u, u^2\}$ by (L3) of Lemma~\ref{L:L+}. 
In particular, $\langle \phi, 1\rangle=0$, 
whence $\phi$ cannot be positive (or negative) throughout $[0,T]$. 
In case $\phi$ positive for $0<|x|<T_1$ and negative for $T_1<|x|<T/2$, for instance,
since $u$ is symmetrically decreasing away from the origin over the interval $(-T/2,T/2)$, we find that 
\[
u(x)-u(T_1)>0\quad\text{for $|x|<T_1$}\quad\text{and}\quad u(x)-u(T_1)<0\quad\text{for $T_1<|x|<T/2$.}
\]
Consequently $\langle \phi, u-u(T_1)\rangle>0$, and $\phi$ cannot be orthogonal to $\{1,u\}$. 
In case $\phi$ changes signs at $x=\pm T_1$ and $x=\pm T_2$, where $T_1<T_2$, correspondingly, 
we find that $(u-u(T_1))(u-u(T_2))$ is positive in $(-T/2,-T_2)\cup(-T_1,T_1)\cup(T_2,T/2)$ 
and negative in $(-T_2,-T_1)\cup(T_1,T_2)$, 
deducing that $\phi$ cannot be orthogonal to $\{1, u, u^2\}$.  
A contradiction therefore proves that $\text{ker}(\delta^2E(u)|_{L^2_{per,even}([0,T])})=\{0\}$.
\end{proof}

\begin{remark}[Power-law nonlinearities]\rm
One may rerun the arguments in the proof of Proposition~\ref{P:nondegeneracy} 
for \eqref{E:KdV'} in the range $0<\alpha\leq 2$ and $0<p<p_{max}$, 
where $p_{max}$ is in \eqref{D:pmax}, 
and establish the nondegenracy of the linearization associated with the traveling wave equation
at a periodic, local constrained minimizer, provided that 
\[
u^{p+1}-\frac{u^{p+1}(T_1)-u^{p+1}(T_2)}{u(T_1)-u(T_2)}u
+\frac{u(T_1)u(T_2)(u^{p}(T_1)-u^{p}(T_2))}{u(T_1)-u(T_2)}
\]
for $T_1<T_2 \in [0,T/2)$ changes its sign at $x=\pm T_1$ and $x=\pm T_2$ 
but nowhere else over the interval $(-T/2,T/2)$. 
Indeed $1,u,u^{p+1}\in\text{range}(\delta^2E(u))$ in place of (L3) of Lemma~\ref{L:L+}, 
but otherwise the proof is identical to that in the case of the quadratic nonlinearity. 
\end{remark}

Unlike in the solitary wave setting, where $n_-(\delta^2E)=1$ at a ground state, 
$\delta^2E$ may have up to two negative eigenvalues at a periodic, local constrained minimizer,
which is characterized by
\begin{align}
n_-(\delta^2E(u;c,a)) = &n_-\left(\begin{matrix}
M_a(u(\cdot\,;c,a)) & P_a(u(\cdot\,;c,a))\\ 
M_c(u(\cdot\,;c,a)) & P_c(u(\cdot\,;c,a)) \end{matrix}\right) \notag \\
=& \#\text{ of sign changes in $1, M_a(u(\cdot\,;c,a)), (M_aP_c-M_cP_a)(u(\cdot\,;c,a))$}, \label{E:n-}
\end{align}
provided that 
\begin{equation}\label{E:N3}
(M_aP_c-M_cP_a)(u(\cdot\,;c,a))\neq 0.
\end{equation} 
A proof based upon ``an index formula" may be found in \cite[Lemma~19]{BHV}. 

\

Notice that \eqref{E:N3} ensures that the mapping $(c,a)\mapsto(P,M)$ is of $C^1$ and locally invertible.
Below we show that it further ensures that the generalized $L^2_{per}([0,T])$-null space of 
the linearized operator associated with \eqref{E:KdV} supports a Jordan block structure, 
which will play a central role in the stability proof in the subsequent section. 

\begin{proposition}[Jordan block structure]\label{P:Jordan}
Let $1/3<\alpha\leq 2$. If $u(\cdot\,;c,a) \in H^{\alpha/2}_{per}([0,T])$ 
for some $c\neq 0$, $a\in \mathbb{R}$ and for some $T>0$ 
locally minimizes $H$ subject to that $P$ and $M$ are conserved and if it satisfies \eqref{E:N3}  
then zero is an $L^2_{per}([0,T])$-generalized eigenvalue 
of the linearized operator associated with \eqref{E:KdV},
\begin{equation}\label{E:LKdV}
J\delta^2E(u(\cdot\,;c,a))=\partial_x(\Lambda^\alpha-2u+c),
\end{equation}
with algebraic multiplicity three and geometric multiplicity two.
\end{proposition}

\begin{proof}
The proof follows from the Fredholm alternative and may be found in \cite[Lemma 6]{BHV};
see \cite{BJK} in the case of generalized KdV equations. Here we merely hit the main points.

Differentiating \eqref{E:H} with respect to $x$, $a$, and $c$, we use \eqref{E:PM} to find that
\[
J\delta^2E(u)u_x=J\delta^2E(u)u_a=0
\quad\text{and}\quad J\delta^2E(u)u_c=-u_x.
\]
Note from Proposition~\ref{P:existence} that $u_a$, $u_c$ are even and $u_x$ is odd. 
Since $\ker(J\delta^2E)$ is at most two dimensional by Proposition~\ref{P:nondegeneracy},
therefore, $\ker(J\delta^2E)=\{u_x,u_a\}$.  By a duality argument and the Fredholm alternative, we then find
that $\ker((J\delta^2E)^\dag)=\{1,u\}$, where the dagger means adjoint.
Thus, if $\phi\in\ker((J\delta^2E(u))^2)/\ker(J\delta^2E(u))$ 
then $J\delta^2E(u)\phi=u_x$ by the Fredholm alternative and \eqref{E:N3}, which, in turn, 
has no solution other than $u_c$ by the Fredholm alternative and \eqref{E:N3}.
\end{proof}

In the solitary wave setting, zero is an $L^2(\mathbb{R})$-eigenvalue of $J\delta^2E(u)$ 
with algebraic multiplicity two and geometric multiplicity one, 
provided that $P_c\neq 0$ at the underlying wave.  
Incidentally a solitary wave corresponds to $a=0$ and $T=+\infty$ in \eqref{E:pKdV},
and hence it depends, up to spatial translations, merely upon the wave speed.
Proposition~\ref{P:Jordan} therefore indicates that \eqref{E:N3} is a natural analogue 
in the periodic wave setting of the familiar condition in the solitary wave setting.

In the case of $\alpha=1$, namely the Benjamin-Ono equation, 
\eqref{E:N3} holds for all periodic traveling waves; see \cite[Section~3.4]{BHV} for the detail.
Concluding the section we shall verify \eqref{E:N3} in the solitary wave limit. 

\begin{lemma}[Solitary wave limit]\label{L:s-limit}
Let $1/2<\alpha\leq 2$. If $u(\cdot\,;c,a,T)$ locally minimizes $H$ in $H^{\alpha/2}_{per}([0,T])$
subject to that $P$ and $M$ are conserved for some $c\neq 0$, $a\in\mathbb{R}$ and $T>0$ then 
\[
M_a(u(\cdot\,;c,a,T))<0\quad\text{and}\quad (M_aP_c-M_cP_a)(u(\cdot\,;c,a,T))>0
\]
for $|a|$ sufficiently small and $T$ sufficiently large.
\end{lemma}

\begin{proof}
The proof may be found in \cite[Lemma~20]{BHV}. Here we include the detail for completeness.

We recall from Remark \ref{R:solitary} that in the range $1/2<\alpha\leq 2$ periodic traveling waves 
of \eqref{E:KdV}, constructed in Proposition \ref{P:existence} as local constrained minimizers, 
tend to the solitary wave as $a\to 0$ and $T\to+\infty$ satisfying $aT\to 0$,
namely in the solitary wave limit, 
which minimizes the Hamiltonian subject to the conservation of the momentum. 
It follows from \eqref{E:scaling} that \eqref{E:pKdV} remains invariant under 
\[ 
u(\cdot\,;c,a,T) \mapsto c^{-1}u(\cdot\,; 1, c^{-2}a, c^{-1/\alpha}T).
\]
Accordingly we may take without loss of generality $c=1$ and we find that 
\[ 
P(1, a, T), M(1, a, T), P_c(1,a,T), M_c(1, a, T)=O(1)
\]
for $|a|$ sufficiently small and $T>0$ sufficiently large; see \cite[Lemma~3.10]{BHV} for the detail. 
Differentiating \eqref{E:I1-KdV} with respect to $a$ and evaluating near the solitary wave limit, 
we use \eqref{E:McPa} to obtain that 
\[ 
M_a(1,a,T)=-T+2M_c(1,a,T)=-T+O(1)<0
\]
for $|a|$ sufficiently small and $T>0$ sufficiently large.
Since an explicit calculation dictates that $P_c(u(\cdot;1,a,T))>0$, furthermore,
\[ 
(M_aP_c-M_cP_a)(1,a,T)=(M_aP_c-M_c^2)(1,a,T)=-P_c(1,a,T)T+O(1)<0
\] 
for $|a|$ sufficiently small and $T>0$ sufficiently large.
\end{proof}

\section{Stability of constrained energy minimizers}\label{S:stability}
We turn the attention to the stability of a periodic, local constrained minimizer for \eqref{E:pKdV} 
with respect to period preserving perturbations.  

Recall from Section \ref{S:existence} that the initial value problem associated with \eqref{E:KdV} 
is well-posed in $X\subset H^{\alpha/2}_{per}([0,T])$ for some interval of time, 
where $H, P, M : X \to \mathbb{R}$ are smooth. 
It suffices to take $X=H^\beta_{per}([0,T])$, $\beta>3/2$. 

Throughout the section let $1/3<\alpha\leq 2$, fixed, and 
let $u_0(\cdot\,, c_0, a_0)\in H^{\alpha/2}_{per}([0,T])$ locally minimize $H$ 
subject to that $P$ and $M$ are conserved 
for some $c_0\neq 0$, $a_0 \in \mathbb{R}$ and for some $T>0$. 
In light of Proposition \ref{P:existence}, 
$u_0 \in X$ and it makes a $T$-periodic, traveling wave of \eqref{E:KdV}. 

Notice that the evolution of \eqref{E:KdV} remains invariant under a one-parameter group 
of isometries corresponding to spatial translations. This motivates us to define the group orbit  
of $u \in X$ as 
\[ 
\mathcal{O}_{u}=\{u(\cdot-x_0):x_0\in\mathbb{R}\}.
\]
Roughly speaking, $u_0(\cdot\,;c_0, a_0)$ is said {\em orbitally stable} if a solution of \eqref{E:KdV} 
remains close to $\mathcal{O}_{u_0}$ under the norm of $X$ for all future times 
whenever the initial datum is sufficiently close to the group orbit of $u_0$ under the norm of $X$. 
We shall elaborate this below in Theorem~\ref{T:stability}. 

\

The present account of orbital stability is inspired by the Lyapunov method. Let
\begin{equation}\label{E:E0}
E_0(u)=H(u)+c_0P(u)+a_0M(u). 
\end{equation}
Proposition \ref{P:existence} implies that $\delta E_0(u_0)=0$, 
i.e., $u_0$ is a critical point of $E_0$.  
Moreover, Proposition \ref{P:nondegeneracy} implies that 
the kernel of $\delta^2 E_0(u_0)$ is spanned by $u_{0x}$. 
Intuitively, $u_0$ is expected to be orbitally stable if $E_0$ is ``convex" at $u_0$. 
As a matter of fact, one may easily verify that if the spectrum of $\delta^2E_0(u_0)$,
except the simple eigenvalue at the origin generated by translation invariance, were positive and bounded away from zero 
then $u_0$ would indeed be orbitally stable.  

However, (L2) of Lemma \ref{L:L+} indicates that 
$\delta^2E_0(u_0)$ admits one or two negative eigenvalues and one zero eigenvalue.
In other words, $u_0$ is a degenerate saddle point of $E_0$ on $X$.
The Lyapunov method therefore may {\em not} be directly applicable. 
In order to control potentially unstable directions and achieve stability, nevertheless,
observe that the evolution under \eqref{E:KdV} does not take place in the entire space $X$, 
but rather on a smooth submanifold of co-dimension two, 
along which the momentum and the mass are conserved. Specifically let
\[
\Sigma_0=\{u\in X: P(u)=P_0,~M(u)=M_0\},
\]
where
\begin{equation}\label{E:P0M0}
P_0=P(u_0(\cdot\,;c_0,a_0))\quad\text{and}\quad M_0=M(u_0(\cdot\,;c_0,a_0)).
\end{equation}
Note that $\mathcal{O}_{u_0} \subset \Sigma_0$ and 
a solution of \eqref{E:KdV} with initial datum in $\Sigma_0$ remains in $\Sigma_0$ at all future times. 
We shall then demonstrate the ``convexity" of $E_0$ on $\Sigma_0$, 
provided that the associated linearization admits a Jordan block structure.


\begin{theorem}[Orbital stability]\label{T:stability}
Let $1/3<\alpha\leq 2$. If $u_0(\cdot;c_0,a_0)\in~H^{\alpha/2}_{per}([0,T])$ 
for some $c_0\neq0$, $a_0\in\mathbb{R}$ and for some $T>0$ 
locally minimizes $H$ subject to that $P$ and $M$ are conserved and if the matrix
\begin{equation}\label{E:matrix}
\left(\begin{matrix}
         M_a(u(\cdot\,;c,a)) & P_a(u(\cdot\,;c,a))\\
         M_c(u(\cdot\,;c,a)) & P_c(u(\cdot\,;c,a))
         \end{matrix}\right)
\end{equation}
is not singular at $u_0(\cdot\,;c_0, a_0)$
then for any $\varepsilon>0$ sufficiently small 
there exists a constant $C=C(\varepsilon)>0$ such that: 

if $\phi\in X$ and $\|\phi\|_X\leq \varepsilon$ 
and if $u(\cdot,t)$ is a solution of \eqref{E:KdV} for some interval of time 
with the initial condition $u(\cdot,0):=u_0+\phi$ 
then $u(\cdot,t)$ may be continued to a solution for all $t>0$ such that
\begin{equation}\label{E:stability}
\sup_{t>0}\inf_{x_0\in\mathbb{R}}\left\|u(\cdot,t)-u_0(\cdot-x_0)\right\|_X\leq C\|\phi\|_X.
\end{equation}
\end{theorem}

The condition that the matrix \eqref{E:matrix} is not singular at the underlying wave ensures that 
the mapping $(c,a)\mapsto (P,M)$ is of $C^1$ and locally invertible. 
In other words, $\Sigma_0$ is nondegenerate.
Moreover it ensures by Proposition~\ref{P:Jordan} that 
the generalized $L^2_{per}([0,T])$-null space of $J\delta^2E_0(u_0(\cdot\,;c_0,a_0))$ 
possesses a Jordan block structure.

At a solitary wave 
$u_0(\cdot\,;c_0)$ of \eqref{E:KdV}, let $P(u)=\int_\mathbb{R} \frac12u^2~dx$ denote the momentum
and note that $\Sigma_0$ consists of all functions such that $P(u)=P(u_0(\cdot\,;c_0))$. 
The condition that \eqref{E:matrix} is not singular then reduces to 
that $P_c(u(\cdot\,;c))\neq 0$ at $u_0(\cdot\,;c_0)$ 
which, in the case of $f(u)=u^{p+1}$, holds for all $p\neq4$.

\

To interpret Theorem~\ref{T:stability}, therefore, a periodic, local constrained minimizer for \eqref{E:pKdV} 
is orbitally stable with respect to period preserving perturbations,
provided that the matrix \eqref{E:matrix} is not singular at the underlying wave
so that the linearized operator associated with \eqref{E:KdV} enjoys a Jordan block structure. 
In particular, nearby solutions need {\em not} be in $\Sigma_0$. 
Note, however, that they are near $\Sigma_0$. 


\

A solitary wave $u_0(\cdot\,;c_0)$ of \eqref{E:KdV} (not necessarily a ground state), in comparison,
was shown in \cite{GSS}, for instance, to be orbitally stable, provided that 
\begin{equation}\label{E:GSSstability} 
\ker(\delta^2E_0(u_0))=\text{span}\{u_{0x}\}, \quad 
n_-(\delta^2E_0(u_0))=1, \quad P_c(u_0(\cdot\,;c_0))>0.
\end{equation}
(Note that the assumption in \cite{GSS} that the symplectic form of a Hamiltonian system be onto 
is dispensable in the proof; see the remark directly following \cite[Theorem 2]{GSS}.)
Conditions in \eqref{E:GSSstability} were, in turn, shown in \cite{BSS} (see also \cite{SS, Ws}) 
to hold if and only if $\alpha>1/2$. In the range $\alpha>1/2$, incidentally,  
a solitary wave of \eqref{E:KdV} minimizes the energy 
subject to the conservation of the momentum; see Remark~\ref{R:solitary}. 
Theorem \ref{T:stability} may therefore be regraded 
as to extend the well-known result about solitary waves.  

In the case of $\alpha>1/2$, recall from  Remark \ref{R:solitary} that 
periodic, local constrained minimizers for \eqref{E:pKdV} are expected 
to tend to the solitary wave as the period increases to infinity, 
which are orbitally stable near the solitary wave limit by Theorem~\ref{T:stability} and Lemma~\ref{L:s-limit},
and the limiting solitary wave is orbitally stable, as well; see \cite{BSS}, for instance. 
In the case of $1/3<\alpha<1/2$, on the other hand, Theorem~\ref{T:stability} indicates that 
orbitally stable, local constrained minimizers for \eqref{E:pKdV} may exist in the periodic wave setting, 
but they are unlikely achieve a limiting wave form with finite energy as the period increases to infinity. 

\

An obvious approach toward Theorem \ref{T:stability} is to rerun the arguments 
in the proof in \cite{GSS} and derive stability criteria, analogous to \eqref{E:GSSstability}; 
see \cite{J1}, for instance, where the last condition in \eqref{E:GSSstability} 
was suitably modified in the case of generalized KdV equations. 
However, it is in general difficult to count the number of negative eigenvalues in the periodic wave setting.
We instead exploit variational properties of the equation --- 
the underlying wave arises as a local constrained minimizer
and the associated linearization is nondegenerate. 
Our proof of Theorem~\ref{T:stability} does not require information about $n_-(\delta^2E_0)$, 
apart from the upper bound in Lemma~\ref{L:L+}. 

\

As a key intermediate step, below we establish the coercivity of $E_0$ on $\Sigma_0$ 
in a neighborhood of the group orbit of $u_0$, 
provided that \eqref{E:matrix} is not singular. 
Let 
\begin{equation}\label{E:tan}
\Sigma_0'=\{\delta P(u_0),\delta M(u_0)\}^\perp
\end{equation}
be the tangent space in $X$ to the sub-manifold $\Sigma_0$ at $u_0$.

\begin{lemma}\label{L:spec}
Under the hypothesis of Theorem \ref{T:stability}, 
\[
\inf\{\left<\delta^2E_0(u_0)v,v\right>:\|v\|_{X}=1,~v\in\Sigma_0',~v\perp u_{0x}\}>0.
\]
\end{lemma}
\begin{proof}
Let $\Pi:L^2_{\rm per}([0,T])\to\Sigma_0'$ be the self-adjoint projection onto $\Sigma_0'$,
and consider 
\[
\Pi\delta^2E_0(u_0):\Sigma_0'\subset X\to\Sigma_0'.
\]
Since $u_0$ locally minimizes $H$, and hence $E_0$, subject to that $P=P_0$ and $M=M_0$, necessarily,
\[
(\Pi\delta^2E_0(u_0))|_{\Sigma_0'}\geq 0.
\] 
Accordingly
\[
\inf\{\left<\delta^2E_0(u_0)v,v\right>:\|v\|_{X}=1,~v\in\Sigma_0',~v\perp \ker(\Pi\delta^2E_0(u_0))\}>0.
\]
The assertion then follows by Proposition \ref{P:nondegeneracy} if 
\begin{equation}\label{E:proj}
\ker(\Pi\delta^2E_0(u_0))=\ker(\delta^2E_0(u_0))={\rm span}\{u_{0x}\}.
\end{equation}

Note that $u_{0x}\in\Sigma_0'$ and $\Pi\delta^2E_0(u_0)u_{0x}=0$.
The kernel of $\Pi\delta^2E_0(u_0)$ is thus at least one dimensional, containing $u_{0x}$. 
Since $1,u_0\in\ker(\delta^2E_0(u_0))^\perp$, and 
$\delta^2E_0(u_0)^{-1}(1)=-\partial_au_0$, $\delta^2E_0(u_0)^{-1}(u_0)=-\partial_cu_0$,
moreover, an ``index formula" (see \cite[Theorem~2.1]{KP}, for instance) implies that 
\[
\dim(\ker(\Pi\delta^2E_0(u_0)))=\dim(\ker(\delta^2E_0(u_0)))+
\left\{\begin{aligned} &1\quad\textrm{if}~~M_aP_c-M_cP_a=0,\\
									&0\quad\textrm{if}~~M_aP_c-M_cP_a\neq 0.
\end{aligned}\right.
\]
This completes the proof, since $\dim(\ker(\Pi\delta^2E_0(u_0)))=1$
by the hypothesis of Theorem~\ref{T:stability}.
\end{proof}

\begin{remark*}\rm
To better understand \eqref{E:proj}, note from \eqref{E:Lu} that 
\[
\Pi\delta^2E_0(u_0)(M_c\partial_au_0-M_a\partial_cu_0)=\Pi(-M_c+M_au_0)=0
\]
and $M_c\partial_au_0-M_a\partial_cu_0\in\Sigma_0'$ 
if and only if \eqref{E:matrix} is singular. 
In other words, $M_c\partial_au_0-M_a\partial_cu_0$ belongs to the $T$-periodic kernel of $\Pi\delta^2E_0(u_0)$ 
(and linearly independent of $u_{0x}$) only if \eqref{E:matrix} is singular.
\end{remark*}

To proceed, we introduce the semi-distance $\rho:X\to\mathbb{R}$, defined by 
\[
\rho(u,v)=\inf_{x_0\in\mathbb{R}}\|u-v(\cdot-x_0)\|_{X}
\] 
and rewrite \eqref{E:stability} as $\sup_{t>0}\rho(u(\cdot, t),u_0)\leq C\|u(\cdot, 0)-u_0\|_X$.
Below we establish the coercivity of $E_0$ on $\Sigma_0$, 
provided that \eqref{E:matrix} is not singular.

\begin{proposition}[Coercivity]\label{P:coercive}
Under the hypothesis of Theorem \ref{T:stability} 
there exist $\varepsilon>0$ and $C=C(\varepsilon)>0$ 
such that if $u\in \Sigma_0$ with $\rho(u,u_0)<\varepsilon$ then 
\begin{equation}\label{E:coercive}
E_0(u)-E_0(u_0)\geq C\rho(u,u_0)^2.
\end{equation}
\end{proposition}

\begin{proof}
The proof closely resembles that of \cite[Theorem 3.4]{GSS} or \cite[Proposition~4.3]{J1}.
Here we include the detail for completeness. 

Throughout the proof and the following, $C$ means a positive generic constant; 
$C$ which appears in different places in the text needs not be the same. 

\

Thanks to the implicit function theorem (see \cite[Lemma 4.1]{BSS}, for instance), 
for $\varepsilon>0$ sufficiently small and for an $\varepsilon$-neighborhood 
$\mathcal{U}_\varepsilon:=\{u\in X:\rho(u,u_0)<\varepsilon\}$ of $\mathcal{O}_{u_0}$
we find a unique $C^1$ map $\omega:\mathcal{U}_\varepsilon\to\mathbb{R}$ such that 
\[
\omega(u_0)=0\quad\text{and}\quad \langle u(\cdot+\omega(u)),u_{0x}\rangle=0
\] 
for all $u\in \mathcal{U}_\varepsilon$. Since $E_0$ is invariant under spatial translations, 
it suffices to show \eqref{E:coercive} along $u(\cdot+\omega(u))$. 
Since $u_0$ locally minimizes $H$, and hence $E_0$, constrained to that $P=P_0$ and $M=M_0$, 
necessarily, 
\begin{equation}\label{E:posE}
\delta^2E_0(u_0)|_{\Sigma_0'}\geq 0,
\end{equation}
where $\Sigma_0'$ is defined in \eqref{E:tan}.
We fix $u\in\mathcal{U}_\varepsilon\cap\Sigma_0$ and write
\begin{equation}\label{E:v}
u(\cdot+\omega(u))=u_0+C_1\delta P(u_0)+\Big(C_2-C_1
\frac{\langle\delta M(u_0),\delta P(u_0)\rangle}{\langle\delta M(u_0),\delta M(u_0)\rangle}\Big) 
\delta M(u_0)+y,
\end{equation}
where $C_1, C_2 \in \mathbb{R}$ and 
$y\in\Sigma_0'\cap\left\{u_{0x}\right\}^\perp$.
Note that $C_1=C_2=y=0$ at $u=u_0$.  

Let $\phi=u(\cdot+\omega(u))-u_0$. We may assume that $\|\phi\|_X<\varepsilon$ 
possibly after replacing $u_0$ by $u_0(\cdot-x_0)$ for some $x_0\in\mathbb{R}$. 
Since $P$ and $M$ remain invariant under spatial translations, Taylor's theorem manifests that
\begin{equation}\label{E:taylor}
\begin{aligned}
P(u)&=P(u(\cdot+\omega(u)))=P(u_0)+\langle\delta P(u_0),\phi\rangle+O(\|\phi\|_{X}^2), \\
M(u)&=M(u(\cdot+\omega(u)))=M(u_0)+\langle\delta M(u_0),\phi\rangle+O(\|\phi\|_{X}^2).
\end{aligned}
\end{equation}
Since $\langle\delta M(u_0),\phi\rangle=C_2\langle\delta M(u_0),\delta M(u_0)\rangle=C_2T$ 
by \eqref{E:v} and \eqref{E:PM}, we infer from the latter equation in \eqref{E:taylor} that 
$C_2=O(\|\phi\|_{X}^2)$. Similarly, since
\begin{align*}
\langle\delta P(u_0),\phi\rangle=&C_1\Big(\langle\delta P(u_0),\delta P(u_0)\rangle
-\frac{\langle\delta M(u_0),\delta P(u_0)\rangle^2}{\langle\delta M(u_0),\delta M(u_0)\rangle}\Big)
+C_2\langle\delta P(u_0),\delta M(u_0)\rangle \\
=&C_1\Big(\|u_0\|_{L^2_{per}(0,T])}^2-\frac{M_0^2}{T}\Big)-C_2M_0,
\end{align*}
the Cauchy-Schwarz inequality, the former equation in \eqref{E:taylor} and \eqref{E:PM} 
lead to that $C_1=O(\|\phi\|_{X}^2)$.

Since $E_0$ remains invariant under spatial translations, furthermore, 
Taylor's theorem manifests that  
\[
E_0(u)=E_0(u(\cdot+\omega(u)))=E_0(u_0)
+\frac{1}{2}\langle\delta^2 E_0(u_0)\phi,\phi\rangle+o(\|\phi\|_{X}^2).
\]
We then use \eqref{E:v} and $C_1, C_2=O(\|\phi\|_X^2)$ to find that 
\[ 
E_0(u)-E_0(u_0)=\frac{1}{2}\langle\delta^2 E_0(u_0)\phi,\phi\rangle+o(\|\phi\|_X^2)
=\frac{1}{2}\langle\delta^2 E_0(u_0)y,y\rangle+O(\|\phi\|_X^2).
\]
Since $y\in\Sigma_0'\cap\left\{u_{0x}\right\}^\perp$, it follows by Lemma~\ref{L:spec} that 
\[
\langle\delta^2 E_0(u_0)y,y\rangle\geq C\|y\|_X^2.
\]
Finally a straightforward calculation reveals that
\begin{align*}
\|y\|_{X}&\geq\Big|\|\phi\|_{X}-\Big\|C_1\delta P(u_0)+
    \Big(C_2-C_1\frac{\langle\delta M(u_0),\delta P(u_0)\rangle}{\langle\delta M(u_0),\delta M(u_0)\rangle}\Big) \delta M(u_0)\Big\|_{X}\Big|\\
&\geq\|\phi\|_{X}-C\|\phi\|_{X}^2,
\end{align*}
whence
\[
E_0(u)-E_0(u_0)\geq C\|\phi\|_{X}^2=C\|u(\cdot+\omega(u))-u_0\|_{X}^2\geq C\rho(u,u_0)^2.
\]
\end{proof}

\begin{proof}[Proof of Theorem \ref{T:stability}]
The proof resembles that of \cite[Lemma 4.1]{J1} in the case of generalized KdV equations.   

\

Let $\varepsilon_0>0$ be such that Proposition \ref{P:coercive} holds and 
let $\phi \in X$ satisfy $\rho(u_0+\phi,u_0)\leq\varepsilon$ 
for some $0<\varepsilon<\varepsilon_0$ sufficiently small. 
By replacing $\phi$ by $\phi(\cdot-x_0)$ for some $x_0\in\mathbb{R}$, if necessary, 
we may assume without loss of generality that  $\|\phi\|_{X}\leq\varepsilon$. 
Since $u_0$ is a critical point of $E_0$, then, Taylor's theorem implies that 
$E_0(u_0+\phi)-E_0(u_0)\leq C\varepsilon^2$.
Furthermore, notice that if $u_0+\phi\in\Sigma_0$ 
then the unique solution $u(\cdot,t)$ of \eqref{E:KdV} with the initial condition $u(\cdot,0)=u_0+\phi$ 
must lie in $\Sigma_0$ as long as the solution exists. 
Since $E_0(u(\cdot, t))=E_0(u(\cdot,0))=E_0(u_0+\phi)$ independently of $t$, on the other hand, 
Proposition \ref{P:coercive} implies that $\rho(u(\cdot,t),u_0)^2\leq C\varepsilon^2$ for all $t\geq 0$. 

\

In case $u_0+\phi$ is not required to be in $\Sigma_0$, 
we utilize the nondegeneracy of the constraint set, i.e., the mapping
\[
(c,a)\mapsto (P(u(\cdot;c,a)), M(u(\cdot;c,a)))
\]
is a period-preserving diffeomorphism from a neighborhood of $(c_0,a_0)$ 
onto a neighborhood of $(P_0, M_0)$. 
We may therefore find $c, a \in\mathbb{R}$ such that $|c|+|a|=O(\varepsilon)$ and 
$u_{\varepsilon}(\cdot\,;c_0+c,a_0+a)$ is a $T$-periodic traveling wave of \eqref{E:KdV} satisfying that
\begin{align*}
P(u_{\varepsilon}(\cdot;c_0+c,a_0+a))=P(u_0+\phi)\quad\text{and}\quad
M(u_{\varepsilon}(\cdot;c_0+c,a_0+a))=M(u_0+\phi).
\end{align*}
Let
\[
{E}_{\varepsilon}(u)=E_0(u)+cP(u)+aM(u).
\]
We may furthermore assume that $u_\varepsilon$ minimizes $E_\varepsilon$ 
subject to that $P$ and $M$ are conserved. 
We then rerun the argument in the proof of Proposition \ref{P:coercive} and show that 
\[
E_{\varepsilon}(u)-E_{\varepsilon}(u_{\varepsilon})\geq C\rho(u,u_{\varepsilon})^2
\]
so long as $\rho(u,u_\varepsilon)$ is sufficiently small.  Since $u_\varepsilon$ is a critical point of 
$E_\varepsilon$, moreover,
$E_\varepsilon(u(\cdot,t))-E_\varepsilon(u_\varepsilon)=
E_\varepsilon(u_0+\phi)-E_\varepsilon(u_\varepsilon) \leq C\varepsilon^2$ 
for all $t\geq 0$. Finally the triangle inequality implies that 
\begin{align*} 
\rho(u(\cdot, t), u_0)^2\leq &C\left(\rho(u(\cdot,t), u_\varepsilon)^2+\rho(u_\varepsilon, u_0)^2\right) \\
\leq &C(E_\varepsilon(u(\cdot,t))-E_\varepsilon(u_\varepsilon))+\|u_\varepsilon-u_0\|_X
\leq C\varepsilon^2
\end{align*}
for all $t\geq 0$. In other words, $u_0(\cdot\,; c_0,a_0)$ is orbitally stable to small perturbations that 
``slightly" change $P$ and $M$.
\end{proof}

One may rerun the arguments in the proof of Theorem \ref{T:stability} mutatis mutandis 
to establish the orbital stability of a periodic, local constrained minimizer for \eqref{E:KdV'}  
in the range $0<\alpha\leq 2$ and $0<p<p_{max}$, where $p_{max}$ is in \eqref{D:pmax},
provided that the matrix \eqref{E:matrix} is not singular at the underlying wave.

\section{Adaptation to equations of regularized long wave type}\label{S:BBM}
The results in Section~\ref{S:existence} through Section~\ref{S:stability} 
are readily adapted to other, nonlinear dispersive equations. 
We shall illustrate this by discussing equations of regularized long wave type
\begin{equation}\label{E:BBM}
u_t-u_x+\Lambda^\alpha u_t+(u^{2})_x=0,
\end{equation}
where $0<\alpha\leq2$.  

In the case of $\alpha=2$, notably, \eqref{E:BBM} recovers the Benjamin-Bona-Mahony (BBM) equation, 
which was advocated in \cite{BBM} as an alternative to the KdV equation.
In fact solutions of the initial value problem associated with the BBM equation 
were argued to enjoy a better smoothness property than those with the KdV equation, 
whereby it was named the regularized long wave equation. For other values of $\alpha$, 
similarly, \eqref{E:BBM} ``regularizes" its KdV counterpart in \eqref{E:KdV}. 
In a small amplitudes and long wavelengths regime, where $u_x+u_t=o(1)$, 
furthermore, \eqref{E:BBM} is formally equivalent to \eqref{E:KdV}. 

The present treatment extends mutatis mutandis to general power-law nonlinearities; 
see Remark \ref{R:gKdV}. We choose to work with the quadratic nonlinearity, 
however, to simplify the exposition.  

\

In the range $\alpha\geq 0$, one may work out the local in time well-posedness for \eqref{E:BBM} 
in $H^\beta_{per}([0,T])$, $\beta> \max(0, (3-\alpha)/2)$, via the energy method,  
corroborating that \eqref{E:BBM} regularizes \eqref{E:KdV}; see Remark~\ref{R:LWP}. 
The proof is rudimentary. Hence we omit the detail. 
With the help of the smoothing effects of $(1-\partial_x^2)^{-1}$, 
the global in time well-posedness for \eqref{E:BBM} may be established in $H^{0+}(\mathbb{R})$
in the case of $\alpha=2$, namely the BBM equation; see \cite{BT}, for instance. 
In the periodic functions setting, however, the existence matter for \eqref{E:BBM} 
seems not adequately understood in spaces of low regularities. 

Throughout the section 
we shall work in a suitable subspace (abusing notation) $X$ of $H^{\alpha/2}_{per}([0,T])$, 
where the initial value problem associated with \eqref{E:BBM} is well-posed for some interval of time;
$T>0$, the period, is fixed. 

\

Notice that \eqref{E:BBM} possesses three conserved quantities (abusing notation)
\begin{align}
H(u)=&\int_{0}^T\Big(\frac12 u^2-\frac{1}{3}u^{3}\Big)~dx\label{E:HBBM} \\
\intertext{and}
P(u)=&\int_{0}^T\frac{1}{2}(u^2+|\Lambda^{\alpha/2}u|^2)~dx,\label{E:PBBM} \\ 
M(u)=&\int_{0}^T u~dx, \label{E:MBBM}
\end{align}
which correspond to the Hamiltonian and the momentum, the mass, respectively. 
Throughout the section we shall use $H$ and $P, M$ 
for those in \eqref{E:HBBM} and \eqref{E:PBBM}, \eqref{E:MBBM}. 
Notice that $H,P,M:X\to\mathbb{R}$ are smooth. 
Notice moreover that \eqref{E:BBM} may be written in the Hamiltonian form 
\[
u_t=J\delta H(u),
\] 
where $J=(1+\Lambda^\alpha)^{-1}\partial_x$. 

\

We seek a periodic traveling wave $u(x,t)=u(x-ct-x_0)$ of \eqref{E:BBM}, 
where $c\in \mathbb{R}$, $x_0 \in \mathbb{R}$ and $u$ is $T$-periodic, satisfying by quadrature that
\begin{equation}\label{E:pBBM}
c(1+\Lambda^\alpha)u+u-u^{2}+a=0
\end{equation}
for some $a\in\mathbb{R}$, or equivalently (abusing notation) 
\begin{equation}\label{E:pBBM'}
\delta E(u;c,a):=\delta(H(u)+cP(u)+aM(u))=0.
\end{equation}
We'll write a periodic traveling wave of \eqref{E:BBM} as $u=u(\cdot\,; c,a)$ with the understanding 
that $T>0$ is arbitrary but fixed and that we may mod out $x_0 \in \mathbb{R}$. 

\

Below we record the existence, symmetry, and regularity properties 
for a family of periodic traveling waves of \eqref{E:BBM}, which arise as local energy minimizers 
subject to conservations of the momentum and the mass. 

\begin{lemma}[Existence, symmetry and regularity]\label{P:BBMexistence}
Let $1/3<\alpha\leq 2$. A local minimizer $u$ for $H$, defined in \eqref{E:HBBM}, subject to that 
$P$ and $M$, defined in \eqref{E:PBBM} and \eqref{E:MBBM}, respectively, are conserved 
exists in $H^{\alpha/2}_{per}([0,T])$ for each $0<T<\infty$ 
and it satisfies \eqref{E:pBBM} for some $c\neq 0$ and $a\in\mathbb{R}$. 
It depends upon $c$ and $a$ in the $C^1$ manner. 
Moreover $u=u(\cdot\,; c,a)$ may be chosen to be even and strictly decreasing 
over the interval $[0, T/2]$, and $u\in H^\infty_{per}([0,T])$. 
\end{lemma}

\begin{proof} 
If $u(x; c, a)$ is $T$-periodic and satisfies \eqref{E:pBBM} 
for some $c\neq 0$ and $a\in \mathbb{R}$ then, by a scaling argument,  
\[
(c+1)u\Big(\Big(\frac{c+1}{c}\Big)^{1/\alpha}x; 1, c^{2}a\Big)
\]
is $\Big(\frac{c+1}{c}\Big)^{1/\alpha}T$-periodic and satisfies
\[
(1+\Lambda^\alpha)u+u-u^2+c^2a=0.
\] 
Therefore it suffices to take $c=1$ in \eqref{E:pBBM}, which brings us to \eqref{E:pKdV}, where $c=2$.  
The proof is then identical to that of Proposition \ref{P:existence}. We omit the detail.
\end{proof}

We promptly address the non\-de\-gen\-er\-a\-cy of the linearization 
associated with \eqref{E:pBBM'} at a periodic, local constrained minimizer for \eqref{E:pBBM}. 

\begin{lemma}[Nondegeneracy]\label{P:nondegeneracy2}
Let $1/3<\alpha\leq 2$. 
If $u(\cdot\,; c,a) \in H^{\alpha/2}_{per}([0,T])$ for some $c \neq 0$, $a \in \mathbb{R}$ 
and for some $T>0$ locally minimizes $H$, defined in \eqref{E:HBBM}, 
subject to that $P$ and $M$, defined in \eqref{E:PBBM} and \eqref{E:MBBM}, respectively, 
are conserved then the associated linearized operator
\begin{equation}\label{E:L+BBM}
\delta^2E(u;c,a)=c\left(1+\Lambda^\alpha\right)+1-2u
\end{equation}
acting on $L^2_{per}([0,T])$ satisfies that
\[ 
\ker(\delta^2E(u;c,a))={\rm span}\{u_x\}.
\] 
\end{lemma}

\begin{proof}
Notice that Lemma~\ref{L:nodal} holds for $\delta^2E(u)$ in \eqref{E:L+BBM};
indeed one may modify the arguments in Appendix~\ref{S:appendix} and prove it.
Notice moreover that (L1) and (L2) of Lemma~\ref{L:L+} hold for \eqref{E:L+BBM}; 
see Section~\ref{S:nondegeneracy} for the detail.

We claim that (L3) of Lemma~\ref{L:L+} holds for $\delta^2E(u)$ in \eqref{E:L+BBM}. 
Differentiating \eqref{E:pBBM'} with respect to $c$ and $a$, respectively, we obtain that
\[ 
\delta^2E(u)u_c=-\delta P(u) \quad \text{and}\quad \delta^2E(u)u_a=-\delta M(u).
\]
Since 
\[
\delta P(u)=(1+\Lambda^\alpha)u\quad\textrm{and}\quad\delta M(u)=1
\]
moreover $1,(1+\Lambda^\alpha)u\in\text{range}(\delta^2 E(u))$. 
Unfortunately $(1+\Lambda^\alpha)u$ may not be strictly monotone over $[0,T/2]$.
Appealing to \eqref{E:pBBM}, on the other hand, we find that
\[
c(1+\Lambda^\alpha)u=u^2-u-a.
\]
Therefore $u^{2}-u\in\text{range}(\delta^2 E)$. Since
\[ 
\delta^2E(u)u=c(1+\Lambda^\alpha)u+u-2u^{2}=-u^{2}-a,
\]
furthermore, $u, u^2\in\text{range}(\delta^2 E)$. This proves the claim. 
The proof is then identical to that of Proposition~\ref{P:nondegeneracy}. 
We omit the detail. \end{proof}

Repeating the arguments in the proofs of Theorem~\ref{T:stability} and Proposition~\ref{P:coercive},
we ultimately establish the orbital stability of a periodic, local constrained minimizer 
for \eqref{E:pBBM}, provided that the associated linearized operator supports a Jordan block structure. 
We summarize the conclusion. 

%

\begin{theorem}[Orbital stability]\label{T:BBMmain}
Let $1/3<\alpha\leq 2$ and $u_0(\cdot\,;c_0,a_0)\in H^{\alpha/2}_{per}([0,T])$ 
for some $c_0\neq 0$, $a_0\in\mathbb{R}$ and for some $T>0$ locally minimizes $H$, 
defined in \eqref{E:HBBM}, subject to that $P$ and $M$, 
defined in \eqref{E:PBBM} and \eqref{E:MBBM}, respectively, are conserved.
If the matrix
\[
\left(\begin{matrix}
         M_a(u(\cdot\,;c,a)) & P_a(u(\cdot\,;c,a))\\
         M_c(u(\cdot\,;c,a)) & P_c(u(\cdot\,;c,a))
         \end{matrix}\right)
\]
is not singular at $u_0(\cdot\,;c_0,a_0)$ then for any $\varepsilon>0$ sufficiently small 
there exists a constant $C=C(\varepsilon)>0$ such that:

if $\phi\in X$ and $\|\phi\|_X\leq\varepsilon$ and if $u(\cdot,t)$ is a solution of \eqref{E:BBM} 
for some time interval with the initial condition $u(\cdot,0)=u_0+\phi$ 
then $u(\cdot,t)$ may be continued to a solution for all $t>0$ such that
\[
\sup_{t>0}\inf_{x_0\in\mathbb{R}}\left\|u(\cdot,t)-u_0(\cdot-x_0)\right\|_X\leq C\|\phi\|_X.
\]
\end{theorem}

Related stability results in the case of $\alpha=1,2$ are found, respectively, in \cite{ASB} and \cite{J3},
among others.

\section{Remark on linear instability}\label{S:instability}
We shall complement the nonlinear stability result in Section~\ref{S:stability}
by discussing the linear instability of periodic traveling waves for the KdV type equation
\begin{equation}\label{E:KdV6}
u_t-\mathcal{M}u_x+f(u)_x=0. 
\end{equation}
Here $\mathcal{M}$ is a Fourier multiplier defined as 
$\widehat{\mathcal{M}u}(\xi)=m(\xi)\widehat{u}(\xi)$, satisfying that
\begin{equation}\label{A:M}
C_1|\xi|^\alpha \leq m(\xi)\leq C_2|\xi|^\alpha, \qquad \text{$|\xi|\gg 1$}
\end{equation}
for some $\alpha\geq 1$ and for some $C_1, C_2>0$, 
while $f:\mathbb{R}\to\mathbb{R}$ is $C^1$, satisfying that
\begin{equation}\label{A:f}
f(0)=f'(0)=0\quad\textrm{and}\quad\lim_{u\to \infty}\frac{f(u)}{u}=\infty.
\end{equation}
Clearly \eqref{E:KdV} fits into the framework. 
We assume that \eqref{E:KdV6} possesses two conserved quantities 
\begin{align*}
P(u)=&\int^T_0 \frac12 u^2~dx
\intertext{and} 
M(u)=&\int^T_0 u~dx, 
\end{align*}
interpreted as the momentum and the mass, respectively.

\

We assume that \eqref{E:KdV6} supports a smooth, four-parameter family of periodic traveling waves, 
denoted $u=u(\cdot-x_0; c, a, T)$, where $c$ and $a$ form an open set in $\mathbb{R}^2$, 
$x_0\in\mathbb{R}$ is arbitrary (and hence we may mod it out), $T_0<T<\infty$ for some $T_0>0$, 
and $u$ is $T$-periodic, satisfying by quadrature that
\begin{equation}\label{E:tKdV1}
\mathcal{M}u-f(u)+cu+a=0. 
\end{equation}
For a broad range of dispersion symbols and nonlinearities, 
the existence of periodic traveling waves of \eqref{E:KdV6} follows from variational arguments, 
e.g., the mountain pass theorem applied to a suitable functional 
whose critical point satisfies \eqref{E:tKdV1}. 

\

Linearizing \eqref{E:KdV6} about a (nontrivial) periodic traveling wave $u=u(\cdot ;c,a,T)$ 
in the frame of reference moving at the speed $c$, we arrive at that 
\begin{equation}\label{E:linear}
v_t=\partial_x(\mathcal{M}-f'(u)+c)v=:\partial_x\mathcal{L}(u;c,a)v.
\end{equation}
Seeking solutions of the form $v(x,t)=e^{\mu t}v(x)$, moreover, we arrive at the spectral problem 
\begin{equation}\label{E:spec}
\mu v=\partial_x\mathcal{L}(u;c,a)v.
\end{equation}
We say that $u$ is {\em linearly unstable} if the $L^2_{per}([0,T])$-spectrum of $\mathcal{L}(u)$ 
intersects the open, right half plane of $\mathbb{C}$. 

\

We shall derive a criterion governing linear instability of periodic traveling waves of \eqref{E:KdV6}, 
which do not necessarily arise as local constrained minimizers. 
In light of Theorem \ref{T:stability}, a local constrained minimizer for \eqref{E:tKdV1} 
is expected nonlinearly stable under the flow induced by \eqref{E:KdV6} under certain assumptions.  

\begin{theorem}[Linear instability]\label{T:instability}
Under the assumptions \eqref{A:M} and \eqref{A:f}, 
let $u=u(\cdot\,;c,a,T)$ be a nontrivial, periodic traveling wave of \eqref{E:KdV6} 
for some $c\neq 0$, $a \in\mathbb{R}$ and for some $T>T_0>0$.
Let $\Pi:L^2_{per}([0,T])\to L^2_{per}([0,T])$ denote the orthogonal projection 
onto the subspace of $L^2_{per}([0,T])$ of mean zero functions, defined by 
\[
\Pi u = u - \frac{1}{T}\int_0^Tu(x)~dx.
\]
Assume that $\Pi \mathcal{L}(u;c,a)$ acting on $\Pi L^2_{per}([0,T])$ satisfies that
\begin{equation}\label{E:non-d}
\ker(\Pi \mathcal{L}(u;c,a))={\rm span}\{u_x\}.
\end{equation}
Then \eqref{E:linear} admits a nontrivial solution of the form $e^{\mu t}v(x)$, 
$v \in H^{\alpha}_{per}([0,T])$ and $\mu >0$, if either
\begin{itemize}
\item[(1)] $n_-(\Pi \mathcal{L}(u;c,a))$ is odd and $P_c(u(\cdot;c,a,T))<0$, or 
\item[(2)] $n_-(\Pi \mathcal{L}(u;c,a))$ is even and $P_c(u(\cdot;c,a,T))>0$.
\end{itemize}
\end{theorem}

Recall that $n_-(\Pi \mathcal{L}(u;c,a))$ is the number of negative eigenvalues of
$\Pi\mathcal{L}(u;c,a)$ acting on $\Pi L^2_{per}([0,T])$.

\

A complete proof is found in \cite{Lin1}, for instance, albeit in the solitary wave setting;  
see also \cite{ABDF} for a Boussinesq equation. 
The arguments in \cite[Section 4]{Lin1} readily extend to the periodic wave setting. 
Here we merely hit the main points.

Notice that \eqref{E:spec} has a nontrivial solution in $H^\alpha_{per}([0,T])$ for some $\mu>0$, 
namely a purely growing mode, if and only if 
\[
A^\mu :=c-\frac{c\partial_x}{\mu-c\partial_x}(\mathcal{M} -f'(u))
\]
has a nontrivial kernel in $H^\alpha_{per}([0,T])$. Since
\begin{align*}
\frac{c\partial_x}{\mu-c\partial_x}&\to 0\qquad \textrm{as }\mu\to +\infty
\intertext{while}
\frac{c\partial_x}{\mu-c\partial_x}&\to \Pi \qquad \textrm{as }\mu\to 0+
\end{align*}
strongly in $L^2_{per}([0,T])$ (see \cite{Lin1} for the detail),
the spectra of $A^\mu$ lie in the right half plane of $\mathbb{C}$ for $\mu>0$ sufficiently large 
while $A^\mu$ converges to $\Pi\mathcal{L}(u)\Pi$ strongly in $L^2_{per}([0,T])$ as $\mu \to 0+$. 
We then examine eigenvalues of $A^\mu$ near the origin in the left half plane of $\mathbb{C}$  
from those of $\Pi\mathcal{L}(u)$ via the moving kernel method. 
Specifically, \eqref{E:non-d} ensures that for $\mu>0$ sufficiently small 
a unique eigenvalue $e_\mu$ of $A^\mu$ exists in the vicinity of the origin 
that depends upon $\mu$ analytically. A lengthy but explicit calculation moreover reveals that
\[ 
\lim_{\mu \to 0+} \frac{e_\mu}{\mu}=0 \quad\text{and}\quad 
\lim_{\mu \to 0+} \frac{e_\mu}{\mu^2}=-P_c(u(\cdot;c,a,T)).
\]
Theorem~\ref{T:instability} therefore follows 
since if $A^\mu$ admits an odd number of eigenvalues in the left half plane of $\mathbb{C}$, 
signaling that the spectrum of $A^\mu$ crosses the origin at some $\mu>0$, 
then a purely growing mode is found.

\

Concluding the section, we shall contrast Theorem~\ref{T:instability} with Theorem~\ref{T:stability}
as it applies to \eqref{E:KdV} near the solitary wave limit. 
It may not be immediately obvious how they complement each other 
since Theorem~\ref{T:stability} is variational in nature whereas 
Theorem~\ref{T:instability} uses spectral information of the associated linearized operator.

\

Below we relate spectral properties of $\Pi\mathcal{L}(u)$ to those of $\mathcal{L}(u)$.

\begin{lemma}[$\Pi\mathcal{L}$ vs. $\mathcal{L}$]\label{L:index}
Let $1/3<\alpha\leq 2$. If $\mathcal{L}(u):=\mathcal{L}(u;c,a)$ is the linearized operator 
associated with \eqref{E:KdV}, which agrees with \eqref{E:L+KdV}, then
\[
n_-(\Pi\mathcal{L}(u))=n_-(\mathcal{L}(u))-
\begin{cases}
1\quad \text{if }M_a\geq 0,\\
0\quad \text{if }M_a<0.\end{cases}
\]
Moreover
\[
\dim(\ker(\Pi\mathcal{L}(u)))=\dim(\ker(\mathcal{L}(u)))+\begin{cases}
1\quad \text{if }M_a= 0,\\
0\quad \text{if }M_a\neq 0.\end{cases}
\]
\end{lemma}

The proof follows from the ``index formula" in \cite[Theorem~2.1]{KP}, for instance. 
We merely note that 
$1\in\ker(\mathcal{L}(u))^\perp$ and $(\mathcal{L}(u))^{-1}1 = -u_a$.  

\

In the case of $1/2<\alpha\leq 2$, we recall from Lemma~\ref{L:s-limit} that 
$M_a(c,a,T)<0$ and $(M_aP_c-M_cP_a)(c,a,T)>0$ 
for $|a|$ sufficiently small and $T>0$ sufficiently large.
Lemma~\ref{L:index}, Proposition~\ref{P:nondegeneracy} and \eqref{E:n-} therefore imply that 
\[
n_-(\Pi \mathcal{L}(u)) = n_-(\mathcal{L}(u)) = 1
\]
near the solitary wave limit.  

In the case of $1/2<\alpha\leq 2$, furthermore, 
Lemma \ref{L:s-limit} and \eqref{E:McPa} dictates that $P_c(c,a,T)>0$
for $|a|$ sufficiently small and $T>0$ sufficiently large.
Theorem~\ref{T:instability} is therefore inconclusive of local constrained minimizers for \eqref{E:pKdV}, 
in the $L^2$-subcritical case, near the solitary wave limit.
This is consistent with the result in Theorem~\ref{T:stability}. 
Indeed one may appeal to \cite{GSS}, for instance, to argue for that 
local constrained minimizers for \eqref{E:pKdV} with large periods and small $a$'s 
are, in the range $\alpha>1/2$, orbitally stable under the flow induced by \eqref{E:KdV}. 

\begin{appendix}

\section{Proof of Lemma~\ref{L:nodal}}\label{S:appendix}

Note that $\Lambda^\alpha$, $0<\alpha<2$, may be viewed as the Dirichlet-to-Neumann operator 
for a suitable {\em local} problem in the periodic half strip $[0,T] \times [0,\infty)$. 
Specifically (see \cite[Theorem 1.1]{RS}, for instance)
\[
C(\alpha)\Lambda^\alpha u:=\lim_{y\to 0+}y^{1-\alpha}w_y(\cdot, y),
\] 
where $w=\mathcal{E}u$ solves the elliptic, boundary value problem
\[
\Delta w+\frac{1-\alpha}{y}w_y=0\quad \text{in $[0,T]_{per} \times (0,\infty)$}, 
\qquad w=u \quad\text{on $[0,T]_{per} \times \{0\}$}
\]
and $C(\alpha)$ is an explicit constant. Accordingly we may derive a variational characterization of 
eigenvalues and eigenfunctions of \eqref{E:L+KdV} in terms of the Dirichlet type functional 
\[ 
\iint_{[0,T]_{per}\times (0,\infty)} |\nabla w(x,y)|^2y^{1-\alpha}~dxdy
+\int_0^T(-2u(x)+c)|w(x,0)|^2~dx
\]
in a suitable function class. 

Note from Proposition~\ref{P:existence} that an eigenfunction $\phi$ of \eqref{E:L+KdV} 
is in $H^{\alpha/2}_{per}([0,T]) \cap C^0_{per}([0,T])$; see also \cite{FL} in the solitary wave setting. 
Similarly, the extension $\mathcal{E}\phi$ belongs to $C^0([0,T]_{per} \times [0,\infty))$.


\

Let $N=\{(x,y)\in[0,T]_{per}\times[0,\infty): \mathcal{E}\phi(x,y)=0\}$, 
which is closed in $[0,T]_{per}\times[0,\infty)$. 
We define the \emph{nodal domains} of $\mathcal{E}\phi$ to be the connected components 
of the open set $[0,T]_{per}\times[0,\infty)\setminus N$ in $[0,T]_{per}\times[0,\infty)$.

\begin{lemma}[Nodal domain bound]\label{L:nodalcount}
Let $0<\alpha<2$. Suppose that \eqref{E:L+KdV} possesses at least $n$ eigenvalues
\[
\lambda_1\leq\lambda_2\leq\ldots \leq \lambda_n.
\]
If $\phi_n\in H^{\alpha/2}_{per}([0,T])\cap C^0_{per}([0,T])$ is a (real) eigenfunction 
of \eqref{E:L+KdV} associated with eigenvalue $\lambda_n$
then its extension $\mathcal{E}\phi_n$ has at most $n$ nodal domains in $[0,T]_{per}\times[0,\infty)$.
\end{lemma}

The proof follows from the nodal domain bound \'a la Courant and 
may be found in \cite[Theorem~3.9]{FL}.  

\begin{proof}[Proof of Lemma~\ref{L:nodal}]
It follows from the Perron-Frobenius argument that eigenvalue $\lambda_1$ is simple and 
a corresponding eigenfunction may be chosen to be strictly positive (or negative) over $[0,T)$. 
Moreover it follows from the arguments of the proof of \cite[Theorem~3.1]{FL} that
an eigenfunction $\phi_2$ associated with eigenvalue $\lambda_2$ 
changes its sign at most twice over $[0,T)$.
This proves the claim for $j=1,2$.
Incidentally $\phi_2$ changes its sign at least once and 
the extension $\mathcal{E}\phi_2$ has at least two nodal domains in $[0,T)\times [0,\infty)$.


\

Let $\phi_3$ denote an eigenfunction associated with eigenvalue $\lambda_3$. 
Suppose that $\phi_3$ changes its sign at least five times on $[0,T)$. 
We then find six points 
\[
0<x_1<\xi_1<x_2<\xi_2<x_3<\xi_3<T
\]
such that, up to multiplication by $-1$, 
\[
\phi_3(x_k)>0 \quad \text{and}\quad \phi_3(\xi_k)<0, \qquad k=1,2,3.
\]
By continuity, moreover,
\[
\mathcal{E}\phi_3(x_k, \varepsilon)>0\quad\text{and}\quad \mathcal{E}\phi_3(\xi_k,\varepsilon)<0,
\qquad k=1,2,3,
\]
and $0\leq\varepsilon\leq\varepsilon_0$ for some $\varepsilon_0$. 
Clearly $\mathcal{E}\phi_3$ has at least two nodal domains in $[0,T)\times [0,\infty)$.
The proof of \cite[Theorem~3.1]{FL}, furthermore, dictates that 
$\mathcal{E}\phi_3$ have at least three nodal domains in $[0,T)\times [0,\infty)$.
For, in the case of exactly two nodal domains, $\phi_3$ cannot change its sign more than twice.
We therefore conclude from Lemma~\ref{L:nodalcount} that 
$\mathcal{E}\phi_3$ has exactly three nodal domains in $[0,T)\times [0,\infty)$.


Since nodal domains are open and connected, and hence pathwise connected, in $[0,T)\times [0,\infty)$,
we may find a continuous curve $\gamma \in C^0([0,1];[0,T]\times [0,\infty))$ such that 
\[
\gamma(0)=x_k,\quad \gamma(1)=x_\ell, \qquad 1\leq k<\ell\leq 3,
\]
and
\[
\mathcal{E}\phi_3(\gamma(t))>0 \qquad \text{for all}\quad t\in [0,1].
\]
In particular, $\gamma(t)$ belongs to the same nodal domain for all $t\in(0,1)$, denoted $\Omega_1$. 
Indeed, if $(x_k,\varepsilon)$'s, $k=1,2,3$ and $0<\varepsilon<\varepsilon_0$, 
belong to separate nodal domains then $\mathcal{E}\phi_3$ has at least four nodal domains,
since $(\xi_k, \varepsilon)$'s belong to different nodal domains.


\

Suppose $k=1$ and $\ell=2$. The Jordan curve theorem implies that 
there cannot be a continuous curve in $[0,T)\times[0,\infty)$ 
which connects $\xi_1$ to either $\xi_2$ or $\xi_3$. 
Therefore $(\xi_1,\varepsilon)$, $0<\varepsilon<\varepsilon_0$, belongs to a second nodal domain,
denoted $\Omega_2$, which is disjoint from the nodal domains 
containing $(\xi_2,\varepsilon)$ and $(\xi_3,\varepsilon)$. 
Since $\mathcal{E}\phi_3$ has exactly three nodal domains, 
$(x_3,\varepsilon)$, $0<\varepsilon<\varepsilon_0$, must belong to the nodal domain $\Omega_1$. 
However, this implies by the Jordan curve theorem that 
$(\xi_2,\varepsilon)$ and $(\xi_3,\varepsilon)$, $0<\varepsilon<\varepsilon_0$, 
must lie in separate nodal domains.
A contradiction proves that $x_1$ and $x_2$ cannot belong to the boundary of the same nodal domain.

To proceed, suppose $k=1$ and $\ell=3$. The Jordan curve theorem similarly implies that
there cannot be a continuous curve in $[0,T)\times[0,\infty)$
which connects $\xi_3$ to either $\xi_1$ or $\xi_2$. 
Therefore, $(\xi_3,\varepsilon)$, $0<\varepsilon<\varepsilon_0$, belongs to a second nodal domain,
which is disjoint from the nodal domains containing $(\xi_1,\varepsilon)$ and $(\xi_2,\varepsilon)$. 
Since $\mathcal{E}\phi_3$ has exactly three nodal domains, 
$(x_2,\varepsilon)$, $0<\varepsilon<\varepsilon_0$, must belong to the nodal domain $\Omega_1$,
which is impossible by the same line of the argument as above.

Lastly, suppose $k=1$ and $\ell=2$. The Jordan curve theorem implies that
there cannot be a continuous curve in $[0,T)\times[0,\infty)$
which connects $\xi_2$ to either $\xi_1$ or $\xi_3$. 
Therefore, $(\xi_2,\varepsilon)$, $0<\varepsilon<\varepsilon_0$, belongs to a second nodal domain,
which is disjoint from the nodal domains containing $(\xi_1,\varepsilon)$ and $(\xi_3,\varepsilon)$. 
This is impossible by the same line of the argument as above. 

A contradiction therefore completes the proof.
\end{proof}

\end{appendix}

\subsection*{Acknowledgements}
VMH is supported by the National Science Foundation 
under grants DMS-1008885 and CAREER DMS-1352597, 
the University of Illinois at Urbana-Champaign under grant RB11162,
an Alfred P. Sloan research fellowship. 
MJ gratefully acknowledges support from the National Science Foundation under grant 
DMS-1211183 and from the University of Kansas General Research Fund under allocation 2302278. 
The authors thank Zhiwu Lin for sharing his report \cite{Lin2},
and the anonymous referees for their careful reading of the manuscript 
and many helpful suggestions and references.

\bibliographystyle{amsalpha}
\bibliography{stabilityBib}

\newcommand{\etalchar}[1]{$^{#1}$}
\providecommand{\bysame}{\leavevmode\hbox to3em{\hrulefill}\thinspace}
\providecommand{\MR}{\relax\ifhmode\unskip\space\fi MR }
\providecommand{\MRhref}[2]{%
  \href{http://www.ams.org/mathscinet-getitem?mr=#1}{#2}
}
\providecommand{\href}[2]{#2}
\begin{thebibliography}{CDLFM07}

\bibitem[AP07]{P}
Jaime Angulo~Pava, \emph{Nonlinear stability of periodic traveling wave
  solutions to the {S}chr\"odinger and the modified {K}orteweg-de {V}ries
  equations}, J. Differential Equations \textbf{235} (2007), no.~1, 1--30.
  \MR{2309564 (2008d:35189)}

\bibitem[APBS06]{PBSM}
Jaime Angulo~Pava, Jerry~L. Bona, and Marcia Scialom, \emph{Stability of
  cnoidal waves}, Adv. Differential Equations \textbf{11} (2006), no.~12,
  1321--1374. \MR{2276856 (2007k:35391)}

\bibitem[APBSO13]{ABDF}
Jaime Angulo~Pava, Carlos Banquet, Jorge~Drumond Silva, and Filipe Oliveira,
  \emph{The {R}egularized {B}oussinesq equation: {I}nstability of periodic
  traveling waves}, J. Differential Equations \textbf{254} (2013), no.~9,
  3994--4023. \MR{3029142}

\bibitem[APN08]{A-P}
Jaime Angulo~Pava and F{\'a}bio M.~A. Natali, \emph{Positivity properties of
  the {F}ourier transform and the stability of periodic travelling-wave
  solutions}, SIAM J. Math. Anal. \textbf{40} (2008), no.~3, 1123--1151.
  \MR{2452883 (2009i:35262)}

\bibitem[ASB11]{ASB}
Jaime Angulo, M{\'a}rcia Scialom, and Carlos Banquet, \emph{The regularized
  {B}enjamin-{O}no and {BBM} equations: well-posedness and nonlinear
  stability}, J. Differential Equations \textbf{250} (2011), no.~11,
  4011--4036. \MR{2776905 (2012a:35264)}

\bibitem[AT91]{AT}
Charles.~J. Amick and John.~F. Toland, \emph{Uniqueness and related analytic
  properties for the {B}enjamin-{O}no equation---a nonlinear {N}eumann problem
  in the plane}, Acta Math. \textbf{167} (1991), no.~1-2, 107--126. \MR{1111746
  (92i:35099)}

\bibitem[BBM72]{BBM}
T.~Brook. Benjamin, Jerry.~L. Bona, and John.~J. Mahony, \emph{Model equations
  for long waves in nonlinear dispersive systems}, Philos. Trans. Roy. Soc.
  London Ser. A \textbf{272} (1972), no.~1220, 47--78. \MR{0427868 (55 \#898)}

\bibitem[Ben67]{Benjamin}
T.~Brooke Benjamin, \emph{Internal waves of permanent form in fluids of great
  depth}, Journal of Fluid Mechanics \textbf{29} (1967), no.~3, 559--592.

\bibitem[Ben72]{Ben-KdV}
T.~Brook Benjamin, \emph{The stability of solitary waves}, Proc. Roy. Soc.
  (London) Ser. A \textbf{328} (1972), 153--183. \MR{0338584 (49 \#3348)}

\bibitem[BH14]{BHV}
J.~C. Bronski and V.~M. Hur, \emph{Modulational instability and variational
  structure}, Stud. Appl. Math. \textbf{132} (2014), no.~4, 285--331.
  \MR{3194028}

\bibitem[BJK11]{BJK}
Jared~C. Bronski, Mathew~A. Johnson, and Todd Kapitula, \emph{An index theorem
  for the stability of periodic travelling waves of {K}orteweg-de {V}ries
  type}, Proc. Roy. Soc. Edinburgh Sect. A \textbf{141} (2011), no.~6,
  1141--1173. \MR{2855892}

\bibitem[Bon75]{Bona}
Jerry~L. Bona, \emph{On the stability theory of solitary waves}, Proc. Roy.
  Soc. London Ser. A \textbf{344} (1975), no.~1638, 363--374. \MR{0386438 (52
  \#7292)}

\bibitem[Bou77]{Bsnesq}
M.~J Boussinesq, \emph{Essai sur la th\'eorie des eaux courants}, M\'emoirs
  pr\'esent\'es par divers savants \'a l'Acad. des Sciences Inst. France
  (s\'erie 2) \textbf{23} (1877), 1--680.

\bibitem[BSS87]{BSS}
Jerry.~L. Bona, Panagiotis.~E. Souganidis, and Walter~A. Strauss,
  \emph{Stability and instability of solitary waves of {K}orteweg-de {V}ries
  type}, Proc. Roy. Soc. London Ser. A \textbf{411} (1987), no.~1841, 395--412.
  \MR{897729 (88m:35128)}

\bibitem[BT09]{BT}
Jerry~L. Bona and Nikolay Tzvetkov, \emph{Sharp well-posedness results for the
  {BBM} equation}, Discrete Contin. Dyn. Syst. \textbf{23} (2009), no.~4,
  1241--1252. \MR{2461849 (2010b:35392)}

\bibitem[CDLFM07]{CFM}
P.~Cardaliaguet, F.~Da~Lio, N.~Forcadel, and R.~Monneau, \emph{Dislocation
  dynamics: a non-local moving boundary}, Free boundary problems, Internat.
  Ser. Numer. Math., vol. 154, Birkh\"auser, Basel, 2007, pp.~125--135.
  \MR{2305351 (2007m:74025)}

\bibitem[CKS{\etalchar{+}}03]{CKSTT}
J.~Colliander, M.~Keel, G.~Staffilani, H.~Takaoka, and T.~Tao, \emph{Sharp
  global well-posedness for {K}d{V} and modified {K}d{V} on {$\Bbb R$} and
  {$\Bbb T$}}, J. Amer. Math. Soc. \textbf{16} (2003), no.~3, 705--749
  (electronic). \MR{1969209 (2004c:35352)}

\bibitem[CT04]{CT}
Rama Cont and Peter Tankov, \emph{Financial modelling with jump processes},
  Chapman \& Hall/CRC Financial Mathematics Series, Chapman \& Hall/CRC, Boca
  Raton, FL, 2004. \MR{2042661 (2004m:91004)}

\bibitem[CW91]{CW}
F.~Michael. Christ and Michael.~I. Weinstein, \emph{Dispersion of small
  amplitude solutions of the generalized {K}orteweg-de {V}ries equation}, J.
  Funct. Anal. \textbf{100} (1991), no.~1, 87--109. \MR{1124294 (92h:35203)}

\bibitem[DK10]{DK}
Bernard Deconinck and Todd Kapitula, \emph{The orbital stability of the cnoidal
  waves of the {K}orteweg-de {V}ries equation}, Phys. Lett. A \textbf{374}
  (2010), no.~39, 4018--4022. \MR{2683991 (2011f:35297)}

\bibitem[DN11]{DN}
Bernard Deconinck and Michael Nivala, \emph{The stability analysis of the
  periodic traveling wave solutions of the m{K}d{V} equation}, Stud. Appl.
  Math. \textbf{126} (2011), no.~1, 17--48. \MR{2724037 (2012a:37159)}

\bibitem[FL13]{FL}
Rupert~L. Frank and Enno Lenzmann, \emph{Uniqueness of non-linear ground states
  for fractional {L}aplacians in {$\Bbb{R}$}}, Acta Math. \textbf{210} (2013),
  no.~2, 261--318. \MR{3070568}

\bibitem[GSS87]{GSS}
Manoussos Grillakis, Jalal Shatah, and Walter~A. Strauss, \emph{Stability
  theory of solitary waves in the presence of symmetry. {I}}, J. Funct. Anal.
  \textbf{74} (1987), no.~1, 160--197. \MR{901236 (88g:35169)}

\bibitem[HJ14]{HJ2}
Vera~Mikyoung Hur and Mathew~A. Johnson, \emph{Modulational instability in the
  {W}hither equation for water waves}, Studies in Applied Mathematics (2014).

\bibitem[HJM15]{HJM}
Vera~Mikyoung Hur, Mathew~A. Johnson, and Jeremy~L. Martin, \emph{Oscillation
  estimates of eigenfunctions}, in preparation (2015).

\bibitem[Hur12]{Hur-breaking}
Vera~Mikyoung Hur, \emph{On the formation of singularities for surface water
  waves}, Commun. Pure Appl. Anal. \textbf{11} (2012), no.~4, 1465--1474.
  \MR{2900797}

\bibitem[Joh09]{J1}
Mathew~A. Johnson, \emph{Nonlinear stability of periodic traveling wave
  solutions of the generalized {K}orteweg-de {V}ries equation}, SIAM J. Math.
  Anal. \textbf{41} (2009), no.~5, 1921--1947. \MR{2564200 (2010k:35412)}

\bibitem[Joh10]{J3}
\bysame, \emph{On the stability of periodic solutions of the generalized
  {B}enjamin-{B}ona-{M}ahony equation}, Phys. D \textbf{239} (2010), no.~19,
  1892--1908. \MR{2684614 (2011h:35248)}

\bibitem[Jos77]{Joseph}
R.~I. Joseph, \emph{Solitary waves in a finite depth fluid}, J. Phys. A
  \textbf{10} (1977), no.~12, 225--227. \MR{0455822 (56 \#14056)}

\bibitem[KdV95]{KdV}
D.~Korteweg and G.~de~Vries, \emph{On the change of form of long waves
  advancing in a rectangular canal, and on a new type of long stationary
  waves}, Phil. Mag. \textbf{39} (1895), 422--443.

\bibitem[KMR11]{KMR}
Carlos~E. Kenig, Yvan Martel, and Luc Robbiano, \emph{Local well-posedness and
  blow-up in the energy space for a class of {$L^2$} critical dispersion
  generalized {B}enjamin-{O}no equations}, Ann. Inst. H. Poincar\'e Anal. Non
  Lin\'eaire \textbf{28} (2011), no.~6, 853--887. \MR{2859931}

\bibitem[KP12]{KP}
Todd Kapitula and Keith Promislow, \emph{Stability indices for constrained
  self-adjoint operators}, Proc. Amer. Math. Soc. \textbf{140} (2012), no.~3,
  865--880. \MR{2869071}

\bibitem[Kwo89]{Kwong}
Man~Kam Kwong, \emph{Uniqueness of positive solutions of {$\Delta u-u+u^p=0$}
  in {${\bf R}^n$}}, Arch. Rational Mech. Anal. \textbf{105} (1989), no.~3,
  243--266. \MR{969899 (90d:35015)}

\bibitem[Lin08]{Lin1}
Zhiwu Lin, \emph{Instability of nonlinear dispersive solitary waves}, J. Funct.
  Anal. \textbf{255} (2008), no.~5, 1191--1224. \MR{2455496 (2010m:35456)}

\bibitem[Lin11]{Lin2}
Zhiwu Lin, \emph{Instability of periodic water waves and dispersive waves},
  talk at the SIAM Conference on Analysis of PDEs (2011).

\bibitem[Mol08]{Molinet}
Luc Molinet, \emph{Global well-posedness in {$L^2$} for the periodic
  {B}enjamin-{O}no equation}, Amer. J. Math. \textbf{130} (2008), no.~3,
  635--683. \MR{2418924 (2009f:35300)}

\bibitem[Ono75]{Ono}
Hiroaki Ono, \emph{Algebraic solitary waves in stratified fluids}, J. Phys.
  Soc. Japan \textbf{39} (1975), no.~4, 1082--1091. \MR{0398275 (53 \#2129)}

\bibitem[Par11]{YP}
Young~Ja Park, \emph{Fractional {P}olya-{S}zeg\"o inequality}, Journal of the
  ChungCheong Mathematical Society \textbf{24} (2011), no.~2, 267--271.

\bibitem[RS12]{RS}
Luz Roncal and Pablo~Ra\'ul Stinga, \emph{Fractional {L}aplacian on the torus},
  arxiv:1209.6104.

\bibitem[SS90]{SS}
Panagiotis~E. Souganidis and Walter~A. Strauss, \emph{Instability of a class of
  dispersive solitary waves}, Proc. Roy. Soc. Edinburgh Sect. A \textbf{114}
  (1990), no.~3-4, 195--212. \MR{1055544 (92a:35143)}

\bibitem[Wei87]{Ws}
Michael~I. Weinstein, \emph{Existence and dynamic stability of solitary wave
  solutions of equations arising in long wave propagation}, Comm. Partial
  Differential Equations \textbf{12} (1987), no.~10, 1133--1173. \MR{886343
  (88h:35107)}

\bibitem[Whi74]{Whitham}
G.~B. Whitham, \emph{Linear and nonlinear waves}, Pure and Applied Mathematics
  (New York), Wiley-Interscience [John Wiley \& Sons], New York, 1974.
  \MR{0483954 (58 \#3905)}

\end{thebibliography}

\end{document}